%% file: thesis.tex
\documentclass{report}
\usepackage{setspace}
\usepackage{epsfig,color,pstricks,pst-poly}
\usepackage{graphicx,amssymb,amsmath,theorem}
\usepackage{mathrsfs}
\usepackage{amscd,latexsym}
\usepackage{tikz}
\usepackage{subfigure}
\usepackage{hyperref}
\usepackage{appendix}

\long\def\killtext#1{}

\newtheorem{theorem}{Theorem}[section]

\newtheorem{definition}[theorem]{Definition}
\newtheorem{lemma}[theorem]{Lemma}
\newtheorem{claim}[theorem]{Claim}
\newtheorem{corollary}[theorem]{Corollary}
\newtheorem{remark}[theorem]{Remark}
\newtheorem{conjecture}[theorem]{Conjecture}
\theorembodyfont{\rmfamily}

\newenvironment{proof}{\medskip\noindent{\bf Proof. }}{\hfill$\square$\medskip}
\newtheorem{proposition}[theorem]{Proposition}

\newcommand{\LL}{{\cal L}}
\newcommand{\G}{{\cal G}}

\newcommand{\x}{{\bf x}}
\newcommand{\Alpha}{\boldsymbol{\alpha}}

\def\Z{\mathbb Z}

\makeatletter
\newcommand\ackname{ACKNOWLEDGMENT}
\if@titlepage
  \newenvironment{acknowledgements}{%
      \titlepage
      \null\vfil
      \@beginparpenalty\@lowpenalty
      \begin{center}%
        \ackname
        \@endparpenalty\@M
      \end{center}}%
     {\par\vfil\null\endtitlepage}
\else
  \newenvironment{acknowledgements}{%
      \if@twocolumn
        \section*{\abstractname}%
      \else
        \small
        \begin{center}%
          {\ackname\vspace{-.1em}\vspace{\z@}}%
        \end{center}%
        \quotation
      \fi}
      {\if@twocolumn\else\endquotation\fi}
\fi
\makeatother

\begin{document}
\large
\thispagestyle{empty}
\begin{doublespacing}
{\centering
OBSTACLES, SLOPES AND TIC-TAC-TOE: AN EXCURSION IN DISCRETE GEOMETRY AND COMBINATORIAL GAME THEORY

By

V S PADMINI MUKKAMALA

A dissertation submitted to the 

Graduate School-New Brunswick

Rutgers, The State University of New Jersey

in partial fulfillment of the requirements

for the degree of 

Doctor of Philosophy

Graduate Program in Mathematics

Written under the direction of

J\'anos Pach and Mario Szegedy

and approved by

\_\_\_\_\_\_\_\_\_\_\_\_\_\_\_\_\_\_\_\_\_\_\_\_\_\_\_\_\_\_\_\_\_\_\_\_\_\_\_\_\_\_\_

\_\_\_\_\_\_\_\_\_\_\_\_\_\_\_\_\_\_\_\_\_\_\_\_\_\_\_\_\_\_\_\_\_\_\_\_\_\_\_\_\_\_\_

\_\_\_\_\_\_\_\_\_\_\_\_\_\_\_\_\_\_\_\_\_\_\_\_\_\_\_\_\_\_\_\_\_\_\_\_\_\_\_\_\_\_\_

\_\_\_\_\_\_\_\_\_\_\_\_\_\_\_\_\_\_\_\_\_\_\_\_\_\_\_\_\_\_\_\_\_\_\_\_\_\_\_\_\_\_\_


New Brunswick, New Jersey

October, 2011

}
\end{doublespacing}
\cleardoublepage
\newpage
\thispagestyle{empty}
{\centering
\vspace{3cm}

To

Amma and Nanna

}
\newpage
\thispagestyle{empty}
\newpage
\pagenumbering{roman}
\pagestyle{plain}
\begin{doublespacing}
{\centering
OBSTACLES, SLOPES AND TIC-TAC-TOE: AN EXCURSION IN DISCRETE GEOMETRY AND COMBINATORIAL GAME THEORY

By

V S PADMINI MUKKAMALA

Dissertation Director: J\'anos Pach and Mario Szegedy

ABSTRACT

}
\end{doublespacing}

\medskip

A drawing of a graph is said to be a {\em straight-line drawing} if the vertices of $G$ are represented by distinct points in the plane and every edge is represented by a straight-line segment connecting the corresponding pair of vertices and not passing through any other vertex of $G$. The minimum number of slopes in a straight-line drawing of $G$ is called the slope number of $G$. 
We show that every cubic graph can be drawn in the plane with straight-line edges using only the four basic slopes $\{0,\pi/4,\pi/2,-\pi/4\}$. We also prove that four slopes have this property if and only if we can draw $K_4$ with them.

Given a graph $G$, an {\em obstacle representation} of $G$ is a
set of points in the plane representing the vertices of $G$,
together with a set of obstacles (connected polygons) such that two vertices
of $G$ are joined by an edge if and only if the corresponding
points can be connected by a segment which avoids all obstacles.
The {\em obstacle number} of $G$ is the minimum number of
obstacles in an obstacle representation of $G$. We show that
there are graphs on $n$ vertices with obstacle number 
$\Omega({n}/{\log n})$.

We show that there is an $m=2n+o(n)$, such that, in the Maker-Breaker game played on $\Z^d$ where Maker needs to put at least $m$ of his marks consecutively in one of $n$ given winning directions, Breaker can force a draw using a pairing strategy.
This improves the result of Kruczek and Sundberg who showed that such a pairing strategy exits if $m\ge 3n$. 
A simple argument shows that $m$ has to be at least $2n+1$ if Breaker is only allowed to use a pairing strategy, thus the main term of our bound is optimal.


\newpage

\begin{acknowledgements}



I would like to thank my parents for always being a pillar of strength for me and supporting me through everything in life and being patient with me despite all my pitfalls. Any acknowledgment would fall short of conveying my gratefulness for having their guidance.

I would also like to thank my brother who has always been the one in whose footsteps I have walked. From school to IIT to Rutgers, he always guided and shaped every vital decision I made. I thank him for always being there.

I would like to thank Rado\v s Radoi\v ci\'c for initiating me into discrete geometry, Mario Szegedy for his valuable guidance as an advisor, and, Jeff Kahn, J\'ozsef Beck, Michel Saks, William Steiger, Doron Zeilberger, Van Vu for teaching me combinatorics and discrete geometry. 

I would also like to thank J\'anos Pach for his patience and guidance as a mentor despite the difficulties I presented as a student, for always striving to encourage me into mathematical research with new problems, for trying to teach me to write papers, and above all, for bringing me to Switzerland where I met my husband.

I would also like to thank all my friends at Rutgers, CUNY and EPFL for making
my stay at all places most enjoyable. 

I would like to thank Linda Asaro at the International Student Center, Rutgers, for being the most helpful advisor and for her punctuality with every request.  

I would also like to thank Pat Barr, Maureen Clausen, Lynn Braun, Demetria Carpenter at Administration, Mathematics Department, for always being helpful.

Lastly, I would like to thank my husband, without whom this PhD would not have been possible and meeting whom would not have been possible without the PhD. His support, guidance and encouragement have shaped not just this PhD, but more aspects of me than I can describe. I thank him for being him and for being there for me.

\end{acknowledgements}

\newpage

\tableofcontents


\newpage






\addcontentsline{toc}{chapter}{List of Figures}
\listoffigures

\newpage
\pagenumbering{arabic}
\pagestyle{plain}
\chapter*{Introduction}
  \addcontentsline{toc}{chapter}{Introduction}
\input{0_intro.tex}

\part{Combinatorial Geometry}
\newpage
\chapter{Slope number}

\input{1_1_slope_introduction.tex}

\input{1_2_subcubicjav.tex}

\input{1_3_paper_revised.tex}

\input{1_4_slope_new_final.tex}

\input{1_5_final_remarks.tex}

\newpage
\chapter{Obstacle number}

\input{2_1_obstacle_introduction.tex}

\input{2_2_obstacleconf.tex}

\input{2_3_obstaclejourn.tex}

\input{2_4_final_remarks.tex}

\newpage
\part{Combinatorial Games}
\newpage
\chapter{Tic-Tac-Toe}

\input{3_1_combinatorial_games_introduction.tex}

\input{3_2_tictactoe.tex}

\newpage

\appendix
\appendixpage
\addappheadtotoc

\input{4_appendix.tex}

\clearpage
\addcontentsline{toc}{chapter}{B Bibliography}
\bibliographystyle{plain}
\bibliography{thesis}
\end{document}

%% file: 0_intro.tex

The field of Graph Theory is said to have first come to light 
with Euler's K\"onigsberg bridge problem in 1736. 
Since then, it has seen much development and also boasts of being a 
subfield of combinatorics that sees intense application. 
In the beginning, graph theory solely comprised of treating a graph as
an abstract combinatorial object. It could even suffice to call it a 
set system with some predefined properties. This outlook was in itself
sufficient to devise and capture some remarkably elegant problems (e.g. 
traveling salesman, vertex and edge coloring, extremal graph theory). 
Besides the large number
of areas it can be applied to (Computer Science, Operations Research,
Game Theory, Decision Theory), 
some independent and naturally intriguing problems
of combinatorial nature were studied in graph theory. 
Around the same time however, a new field in graph theory arose. 
It can be said that, because of the simplicity of representing so many 
things as graphs, the natural question of the simplicity of representing
graphs themselves, on paper or otherwise, came up. 
Thus started the yet nascent field of graph
drawing, where now the concern was mostly of representing a graph in the plane
and in particular, how simply can it be represented.

This idea in itself had many far reaching applications. Among the first of 
it was the four color theorem, where the simplicity of drawing 
maps was the concern. Since then, many more questions have arisen. An important
one of these was drawing graphs in the plane without crossings, or in other words,
estimating the crossing number of graphs. Although initially the idea of edges was restricted to being
topological curves, before long, a natural further restriction 
was introduced.
To discretize the problem further and to add to the aspect of naturally 
representing graphs, the branch of straight-line drawings of graphs started.
With straight-line drawings, besides the old questions like crossing 
number etc., some new, purely geometrical concerns arose. For example, one 
way of simplifying a drawing of a graph could be to try to reduce the number
of slopes used in the drawing. This led to the general notion 
(introduced by Wade and Chu) of the slope number, which for a graph is
the minimum number of slopes required to draw it.

The slope number of graphs is at least half their maximum degree. This lead to
the intuitive belief that bounded degree graphs might allow for small slope
numbers. This was shown to the contrary, with a counting argument, that 
even graphs with maximum degree at most five need not have a bounded slope
number. Graphs with maximum degree two can trivially be shown to require 
at most three slopes. This restricts our attention to graphs with maximum
degree three and four. Maximum degree four, still, to the best of our knowledge,
remains an exciting open problem. For maximum degree three although, 
using some previous results, we can provide an exact answer. We show that 
four slopes, even the four fixed slopes of North, East, Northeast, Northwest
are sufficient to draw all graphs with maximum degree three. Since $K_4$
requires at least four slopes in the plane, this indeed is an exact answer.

\medskip

Another interesting notion about straight-line graphs, that arises in 
many natural contexts is that of representing it as a visibility graph.
Given a set of points and a set of polygons (obstacles) in the plane,
a visibility graph's edges comprise of exactly all mutually visible
vertex pairs. Visibility graphs have numerous applications in Computer
Science (Vision, Graphics, Robot motion planning).
A natural question that arises from considering the simplicity of
such a representation is 
to find the smallest number of obstacles one has to
use in the plane to represent a graph. This is defined as the obstacle
number of the graph. 
It was shown that there are graphs
on $n$ vertices that require $\Omega(\sqrt{\log n})$ obstacles. 
We improve this to show that there are graphs which require 
$\Omega(\frac{n}{\log n})$ obstacles, which can be further improved 
for nicer obstacles. In particular, we show that there are graphs
which require $\Omega(\frac{n^2}{\log n})$ segment obstacles. 


\medskip

The final part of the thesis deals with 
positional games. Many 
combinatorial games (Tic-Tac-Toe, hex, Shannon switching game) 
can be thought of as played on a
hypergraph in which a point is claimed by one of the two players at every turn. 
Winning in such a game is characterized by the capture of a ``winning set''
by a player. All the winning sets form the edges in our hypergraph. 
Such games are called positional games.
If the 
second player wins if there is a draw, 
then the game is called a {\em Maker-Breaker} game, and the players are called
respectively, {\em Maker} and
{\em Breaker}. We may also note that if the Breaker can find a pairing 
of the vertices such that every winning set contains a pair, then he can
achieve a draw, called a {\em pairing strategy draw}. 

The classical Tic-Tac-Toe game can be generalized to the hypergraph $\Z^d$
with winning sets as consecutive $m$ points in $n$ given directions.
For example, in the Five-in-a-Row game $d=2$, $m=5$ and $n=4$, the winning directions are the vertical, the horizontal and the two diagonals with slope $1$ and $-1$.
It was shown by Hales and Jewett, that for the four above given directions of the two dimensional grid and $m=9$ the second player can achieve a pairing strategy draw. In the general version, it was shown by Kruczek and Sundberg that the second player has a pairing strategy if $m\ge 3n$ for any $d$. 
They conjectured that there is always a pairing strategy for $m\ge 2n+1$, generalizing the result of Hales and Jewett.
We show that their conjecture is asymptotically true, i.e. 
for $m=2n+o(n)$.
In fact we prove the stronger result where $m-1 = p \ge 2n+1$, $p$ a prime. 
This is indeed stronger because there is a prime between $n$ and $n+o(n)$.

%% file: 1_1_slope_introduction.tex
\section{Introduction}\label{slope_intro}
A {\em straight-line drawing} of a graph, $G$, in the plane is obtained if the vertices of $G$ are represented by distinct points in the plane and every edge is represented by a straight-line segment connecting the corresponding pair of vertices and not passing through any other vertex of $G$.
If it leads to no confusion, in notation and terminology we make no distinction between a vertex and the corresponding point, and between an edge and the corresponding segment.
The {\em slope} of an edge in a straight-line drawing is the slope of the corresponding segment.
Wade and Chu \cite{wc94} defined the {\em slope number}, $sl(G)$, of a graph $G$ as the smallest number $s$ with the property that $G$ has a straight-line drawing with edges
of at most $s$ distinct slopes.

Our terminology is somewhat unorthodox: by the {\em slope} of a
line $\ell$, we mean the angle $\alpha$ modulo $\pi$ such that a
counterclockwise rotation through $\alpha$ takes the $x$-axis to a
position parallel to $\ell$. The slope of an edge (segment) is the
slope of the line containing it. In particular, the slopes of the
lines $y=x$ and $y=-x$ are $\pi/4$ and $-\pi/4$, and they are
called {\em Northeast} (or Southwest) and {\em Northwest} (or
Southeast) lines, respectively.
Directions are often
abbreviated by their first letters: N, NE, E, SE, etc. These four
directions are referred to as {\em basic}. That is, a line $\ell$
is said to be of one of the four basic directions if $\ell$ is
parallel to one of the axes or to one of the NE and NW lines $y=x$
and $y=-x$.

Obviously, if $G$ has a vertex of degree $d$, then its slope number is at least
$\lceil d/2\rceil$. Dujmovi\'c et al.~\cite{dsw04} asked if the slope number of a graph with bounded maximum degree $d$ could be arbitrarily large. Pach and P\'alv\"olgyi \cite{pp06} and Bar\'at, Matou\v sek, Wood \cite{bmw06} (independently) showed with a counting argument that the answer is yes for $d\ge 5$.

In \cite{kppt08_2}, it was shown that cubic ($3$-regular) graphs could be drawn with five slopes. The major result from which this was concluded was that subcubic graphs\footnote{A graph is subcubic if it is a proper subgraph of a cubic graph, i.e. the degree of every vertex is at most three and it is not cubic (not $3$-regular).} can be drawn with the four basic slopes. 
We note here that the proof of this was slightly incorrect. We give below
a stronger version of that theorem, in which the shortcomings of the incorrect
proof can be overcome. Before the statement of the theorem, we clarify
some terminology used in it.

For any two points $p_1=(x_1,y_1), p_2=(x_2,y_2)\in {\bf R}^2$, we
say that $p_2$ is {\em to the North} (or {\em to the South} of $p_1$ if
$x_2=x_1$ and $y_2>y_1$ (or $y_2<y_1$). Analogously, we say that
$p_2$ is to the Northeast (to the Northwest) of $p_1$ if $y_2>y_1$
and $p_1p_2$ is a Northeast (Northwest) line. 

\begin{theorem}\label{slopenum2} Let $G$ be a graph without components that are cycles
and whose every vertex has degree at most three.
Suppose that $G$ has at least one vertex of degree less than
three, and denote by $v_1, ..., v_m$ the vertices of degree at most two $\;(m\ge 1)$.

Then, for any sequence $x_1, x_2, \ldots , x_n$ of real numbers,
linearly independent over the rationals, $G$ has a straight-line
drawing with the following properties:

\noindent(1) {\em Vertex $v_i$ is mapped into a point with
$x$-coordinate $x(v_i)=x_i\; (1\le i\le m)$;}

\noindent(2) {\em The slope of every edge is $0, \pi/2, \pi/4,$ or
$-\pi/4.$}

\noindent(3) {\em No vertex is to the North of any vertex of
degree {\it two}.}

\noindent(4) {\em No vertex is to the North or to the Northwest of
any vertex of degree {\it one}.}

\noindent(5) {\em The $x$-coordinates of all the vertices are a linear combination with rational coefficients of $x_1,\ldots,x_n$.}

\end{theorem}

Therefore, cubic graphs require one additional slope and hence, five slopes. We improve this as following.

\begin{theorem}\label{thm_4slopes}
Every connected cubic graph has a straight-line drawing with only four slopes.
\end{theorem}

The above theorem gives a drawing of connected cubic graphs with four slopes, 
one of which is not a basic slope. Further, for disconnected cubic graphs, 
we require $5$ slopes. 
We show a reduction of cubic graphs with triangles (Lemma \ref{triangle_free}) 
because of which, instead of the above theorem, our focus will be to prove the following. 

\begin{theorem}\label{subthm_4slopes}
Every triangle-free connected cubic graph has a straight-line drawing with only four slopes.
\end{theorem}


We note that the four slopes used above are not the four basic slopes. 
Towards this, it
was shown by Max Engelstein \cite{eng} that $3$-connected cubic graphs with a Hamiltonian cycle can be drawn with the four basic slopes.

We later improve all these results by the following

\begin{theorem}\label{thm_4basic_slopes}
Every cubic graph has a straight-line drawing with only the four basic slopes.
\end{theorem}

\begin{figure*}[htp]
{\centering
\subfigure[Petersen graph]{
\begin{tikzpicture}[scale=1]
\node [fill=black,circle,inner sep=1pt] (1) at (1,0) {}; 
\node [fill=black,circle,inner sep=1pt] (2) at (0,1) {}; 
\node [fill=black,circle,inner sep=1pt] (3) at (2,1) {}; 
\node [fill=black,circle,inner sep=1pt] (4) at (4,1) {}; 
\node [fill=black,circle,inner sep=1pt] (5) at (3,0) {}; 
\node [fill=black,circle,inner sep=1pt] (6) at (0,2) {}; 
\node [fill=black,circle,inner sep=1pt] (7) at (1,3) {}; 
\node [fill=black,circle,inner sep=1pt] (8) at (2,3) {}; 
\node [fill=black,circle,inner sep=1pt] (9) at (3,3) {}; 
\node [fill=black,circle,inner sep=1pt] (10) at (4,2) {}; 
\draw [black] (1) -- (2) -- (3) -- (4) -- (5) -- (1) -- (7) -- (8) -- (9) -- (10) -- (6) -- (7);
\draw [black] (2) -- (6);
\draw [black] (3) -- (8);
\draw [black] (4) -- (10);
\draw [black] (5) -- (9);
\end{tikzpicture}
}
\qquad
\subfigure[$K_{3,3}$]{
\begin{tikzpicture}[scale=1]
\node [fill=black,circle,inner sep=1pt] (1) at (1,0) {}; 
\node [fill=black,circle,inner sep=1pt] (2) at (3,0) {}; 
\node [fill=black,circle,inner sep=1pt] (3) at (4,1) {}; 
\node [fill=black,circle,inner sep=1pt] (4) at (3,2) {}; 
\node [fill=black,circle,inner sep=1pt] (5) at (1,2) {}; 
\node [fill=black,circle,inner sep=1pt] (6) at (0,1) {}; 
\draw [black] (1) -- (2) -- (3) -- (4) -- (5) -- (6) -- (1);
\draw [black] (1) -- (4);
\draw [black] (2) -- (5);
\draw [black] (3) -- (6);
\end{tikzpicture}
}
\caption[The Petersen graph and $K_{3,3}$]{The Petersen graph and $K_{3,3}$ drawn with the four basic slopes.}
} 
\label{fig:petersen}
\end{figure*}
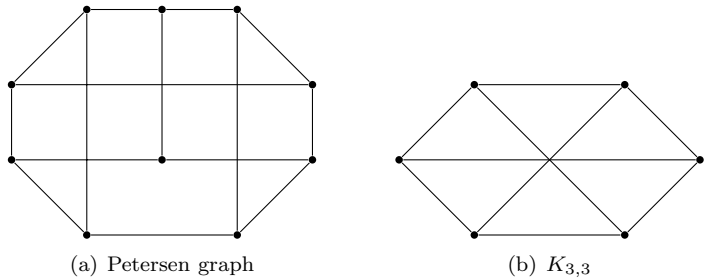

This is the first result about cubic graphs that uses a nice, fixed set of slopes instead of an unpredictable set, possibly containing slopes that are not rational multiples of $\pi$.
Also, since $K_4$ requires at least $4$ slopes, this settles the question of determining the minimum number of slopes required for cubic graphs.

We also prove 

\begin{theorem}\label{karakterizacio}
Call a set of slopes {\em good} if every cubic graph has a straight-line drawing with them. 
Then the following statements are equivalent for a set $S$ of four slopes.
\begin{enumerate}
\item $S$ is good.
\item $S$ is an affine image of the four basic slopes.
\item We can draw $K_4$ with $S$.
\end{enumerate}
\end{theorem}

The problem whether the slope number of graphs with maximum degree four is unbounded or not remains an interesting open problem.


There are many other related graph parameters. 
The {\em thickness} of a graph $G$ is defined as the smallest number of planar subgraphs it can be decomposed into \cite{MuOS}. It is one of the several widely known graph parameters that measures how far $G$ is from being planar.
The {\em geometric thickness} of $G$, defined as the smallest number of {\em crossing-free} subgraphs of a straight-line drawing of $G$ whose union is $G$, is another similar notion \cite{Ka}.
It follows directly from the definitions that the thickness of any graph is at most as
large as its geometric thickness, which, in turn, cannot exceed its slope number.
For many interesting results about these parameters, consult \cite{DiEH, dek04, dsw04, DuW, E04, HuSW}.

A variation of the problem arises if (a) two vertices in a drawing have an edge between them if and only if the slope between them belongs to a certain set $S$ and, (b) collinearity of points is allowed. This violates the condition stated before that an edge cannot pass through vertices other than its end points. For instance, $K_n$ can be drawn with one slope. The smallest number of slopes that can be used to represent a graph in such a way is called the {\em slope parameter} of the graph.
Under these set of conditions, \cite{ambaha06} proves that the slope parameter of subcubic outerplanar graphs is at most $3$.
It was shown in \cite{kppt10} that the slope parameter of every cubic graph is at most seven.
If only the four basic slopes are used, then the graphs drawn with the above conditions are called Queen's graphs and \cite{amba06} characterizes certain graphs as Queen's graphs. Graph theoretic properties of some specific Queen's graphs can be found in \cite{bs09}.

Another variation for planar graphs is to demand a planar drawing. The {\em planar slope number} of a planar graph is the smallest number of distinct slopes with the property that the graph has a straight-line drawing with non-crossing edges using only these slopes.
Dujmo\-vi\'c, Eppstein, Suderman, and Wood \cite{dsw07} raised the question whether there exists a function $f$ with the property that the planar slope number of every planar graph with maximum degree $d$ can be bounded from above by $f(d)$.
Jelinek et al.~\cite{JJ10} have shown that the answer is yes for {\em outerplanar} graphs, that is, for planar graphs that can be drawn so that all of their vertices lie on the outer face.
Eventually the question was answered in \cite{kpp10} where it was proved that any bounded degree planar graph has a bounded planar slope number.

Finally we would mention a slightly related problem. Didimo et al.~\cite{Didimo} studied drawings of graphs where edges can only cross each other in a right angle. Such a drawing is called an RAC (right angle crossing) drawing. They showed that every graph has an RAC drawing if every edge is a polygonal line with at most three bends (i.e. it consists of at most four segments). They also gave upper bounds for the maximum number of edges if less bends are allowed. Later Arikushi et al.~\cite{rado} showed that such graphs can have at most $O(n)$ edges. Angelini et al.~\cite{Angelini} proved that every cubic graph admits an RAC drawing with at most one bend. It remained an open problem whether every cubic graph has an RAC drawing with straight-line segments. If besides orthogonal crossings, we also allow two edges to cross at $45^\circ$, then it is a straightforward corollary of Theorem \ref{thm_4basic_slopes} that every cubic graph admits such a drawing with straight-line segments.


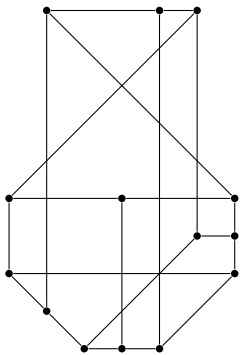
\begin{figure*}[htp]
{\centering
\begin{tikzpicture}[scale=0.5]
\node [fill=black,circle,inner sep=1pt] (1) at (0,2) {}; 
\node [fill=black,circle,inner sep=1pt] (2) at (1,1) {}; 
\node [fill=black,circle,inner sep=1pt] (3) at (2,0) {}; 
\node [fill=black,circle,inner sep=1pt] (4) at (3,0) {}; 
\node [fill=black,circle,inner sep=1pt] (5) at (4,0) {}; 
\node [fill=black,circle,inner sep=1pt] (6) at (6,2) {}; 
\node [fill=black,circle,inner sep=1pt] (7) at (6,3) {}; 
\node [fill=black,circle,inner sep=1pt] (8) at (5,3) {}; 
\node [fill=black,circle,inner sep=1pt] (9) at (0,4) {}; 
\node [fill=black,circle,inner sep=1pt] (10) at (3,4) {}; 
\node [fill=black,circle,inner sep=1pt] (11) at (6,4) {}; 
\node [fill=black,circle,inner sep=1pt] (12) at (1,9) {}; 
\node [fill=black,circle,inner sep=1pt] (13) at (4,9) {}; 
\node [fill=black,circle,inner sep=1pt] (14) at (5,9) {}; 
\draw [black] (1) -- (2) -- (3) -- (4) -- (5) -- (6) -- (1);
\draw [black] (3) -- (8) -- (7) -- (6);
\draw [black] (9) -- (10) -- (11) -- (12) -- (13) -- (14) -- (9);
\draw [black] (1) -- (9);
\draw [black] (2) -- (12);
\draw [black] (4) -- (10);
\draw [black] (5) -- (13);
\draw [black] (7) -- (11);
\draw [black] (8) -- (14);
\end{tikzpicture}
\caption[The Heawood graph]{The Heawood graph drawn with the four basic slopes.}
} 
\label{fig:heawood}
\end{figure*}

%% file: 1_2_subcubicjav.tex
\section{Correct Proof of the Subcubic Theorem}

We would like the reader to note that this is a modification
of the proof as it appears in \cite{kppt08_2}.

Note that it is enough to establish the theorem for connected graphs, because if the different components of $G$ are drawn separately and placed far above each other, then none of the properties will be violated.

\subsection{Embedding Cycles}

Let $C$ be a straight-line drawing of a cycle in the plane. A
vertex $v$ of $C$ is said to be a {\em turning point} if the
slopes of the two edges meeting at $v$ are not the same.

We start with two simple auxiliary statements.

\begin{lemma}\label{slopenumlem21} Let $C$ be a straight-line drawing
of a cycle such that the slope of every edge is $0$, $\pi/4$, or
$-\pi/4$. Then the $x$-coordinates of the vertices of $C$ are {\it
not} independent over the rational numbers.

Moreover, there is a vanishing linear combination of the
$x$-coordinates of the vertices, with 
as many nonzero
(rational) coefficients as many turning points $C$ has.
\end{lemma}

\noindent{\bf Proof.} Let $v_1, v_2,\ldots, v_n$ denote the
vertices of $C$ in cyclic order ($v_{n+1}=v_1$). Let $x(v_i)$ and
$y(v_i)$ be the coordinates of $v_i$. For any $i\; (1\le i\le n)$,
we have
$y(v_{i+1})-y(v_i)=\lambda_i\left(x(v_{i+1})-x(v_i)\right),$ where
$\lambda_i=0,1,$ or $-1$, depending on the slope of the edge
$v_iv_{i+1}$. Adding up these equations for all $i$, the left-hand
sides add up to zero, while the sum of the right-hand sides is a
linear combination of the numbers $x(v_1), x(v_2),\ldots, x(v_n)$
with integer coefficients of absolute value at most {\em two}.

Thus, we are done with the first statement of the lemma, unless
all of these coefficients are zero. Obviously, this could happen
if and only if $\lambda_1=\lambda_2=\ldots=\lambda_n$, which is
impossible, because then all points of $C$ would be collinear,
contradicting our assumption that in a proper {\em straight-line
drawing} no edge is allowed to pass through any vertex other than
its endpoints.

To prove the second statement, it is sufficient to notice that the
coefficient of $x(v_i)$ vanishes if and only if $v_i$ is not a
turning point. \hfill $\Box$
\medskip

Lemma \ref{slopenumlem21} shows that Theorem \ref{slopenum2} does not hold if $G$ is a cycle.
Nevertheless, according to the next claim, cycles satisfy a very
similar condition. Observe, that the main difference is that here 
we have an exceptional vertex, denoted by $v_0$.

\begin{lemma}\label{slopenumlem22} Let $C$ be a cycle with vertices
$v_0, v_1, \ldots , v_m$, in this cyclic order.

Then, for any real numbers $x_1, x_2, \ldots , x_m$, linearly
independent over the rationals, $C$ has a straight-line drawing
with the following properties:

\noindent(1) {\em Vertex $v_i$ is mapped into a point with
$x$-coordinate $x(v_i)=x_i\; (1\le i\le m)$;}

\noindent(2) {\em The slope of every edge is $0, \pi/4,$ or
$-\pi/4.$}

\noindent(3) {\em No vertex is to the North of any other vertex.}

\noindent(4) {\em No vertex has a larger $y$-coordinate than
$y(v_0)$.}

\noindent(5) {\em The $x$-coordinate of $v_0$ is a linear combination with rational coefficients of $x_1,\ldots,x_m$.}
\end{lemma}

\noindent {\bf Proof.}
We can assume without loss of generality that $x_2>x_1$.
Place $v_1$ at the point $(x_1,0)$ of the
$x$-axis. Assume that for some $i<m$, we have already determined
the positions of $v_1, v_2, \ldots v_{i}$, satisfying conditions
(1)--(3). If $x_{i+1}>x_i$, then place $v_{i+1}$ at the (unique)
point Southeast of $v_i$, whose $x$-coordinate is $x_{i+1}$. If
$x_{i+1}<x_i$, then put $v_{i+1}$ at the point West of $x_i$,
whose $x$-coordinate is $x_{i+1}$. Clearly, this placement of
$v_{i+1}$ satisfies (1)--(3), and the segment $v_iv_{i+1}$ does
not pass through any point $v_j$ with $j<i$.

After $m$ steps, we obtain a noncrossing straight-line drawing of
the path $v_1v_2\ldots v_{m}$, satisfying conditions (1)--(3). We
still have to find a right location for $v_0$. Let $R_W$ and
$R_{SE}$ denote the rays (half-lines) starting at $v_1$ and
pointing to the West and to the Southeast. Further, let $R$ be the
ray starting at $v_m$ and pointing to the Northeast. It follows
from the construction that all points $v_2, \ldots, v_{m}$ lie in
the convex cone below the $x$-axis, enclosed by the rays $R_W$ and
$R_{SE}$.

Place $v_0$ at the intersection point of $R$ and the
$x$-axis. Obviously, the segment $v_mv_0$ does not pass through
any other vertex $v_j\;(0<j<m)$. Otherwise, we could find a
drawing of the cycle $v_jv_{j+1}\ldots v_m$ with slopes $0,
\pi/4,$ and $-\pi/4$. By Lemma \ref{slopenumlem21}, this would imply that the
numbers $x_j, x_{j+1}, \ldots, x_m$ are {\em not} independent over
the rationals, contradicting our assumption. It is also clear that
the horizontal segment $v_0v_1$ does not pass through any vertex
different from its endpoints because all other vertices are
below the horizontal line determined by $v_0v_1$.
Hence, we obtain a proper
straight-line drawing of $C$ satisfying conditions (1),(2), and
(4). Note that (5) automatically follows from Lemma \ref{slopenumlem21}.

It remains to verify (3). The only thing we have to check is that
$x(v_0)$ does not coincide with any other $x(v_i)$. Suppose it
does, that is, $x(v_0)=x(v_i)=x_i$ for some $i>0$. By the second
statement of Lemma \ref{slopenumlem21}, there is a vanishing linear combination
$$\lambda_0x(v_0)+\lambda_1x_1+\lambda_2x_2+\ldots+\lambda_mx_m=0$$
with rational coefficients $\lambda_i$, where the number of
nonzero coefficients is at least the number of turning points,
which cannot be smaller than {\em three}. Therefore, if in this
linear combination we replace $x(v_0)$ by $x_i$, we still obtain a
nontrivial rational combination of the numbers $x_1, x_2,\ldots,
x_m$. This contradicts our assumption that these numbers are
independent over the rationals. \hfill $\Box$
\medskip

\subsection{Subcubic Graphs - Proof of Theorem \ref{slopenum2}}

First we settle Theorem \ref{slopenum2} in a special case.
 
\begin{lemma}\label{slopenumlem31} Let $m,k\ge 2$ and let $G$ be a graph consisting of two
disjoint cycles, $C=\{v_0, v_1, \ldots , v_m\}$ and $C'=\{v_0', v_1', \ldots , v_m'\}$,
connected by a single edge $v_0v'_0$. 

Then, for any sequence $x_1, x_2, \ldots , x_m, x'_1, x'_2, \ldots , x'_k$  
of real numbers, linearly independent over the rationals, $G$ has a 
straight-line drawing satisfying the following conditions:

\noindent(1) {\em The vertices $v_i$ and $v'_j$ are mapped into points with
$x$-coordinates $x(v_i)=x_i\; (1\le i\le m)$ and $x(v_j)=x'_j\; (1\le j\le k)$.}

\noindent(2) {\em The slope of every edge is $0, \pi/2, \pi/4,$ or $-\pi/4.$}

\noindent(3) {\em No vertex is to the North of any vertex of degree {\it two}.}

\noindent(4) {\em The $x$-coordinates of all the vertices are a linear combination with rational coefficients of $x_1, x_2, \ldots , x_m, x'_1, x'_2, \ldots , x'_k$.}  

\end{lemma}

\noindent {\bf Proof.}
Apply Lemma \ref{slopenumlem22} to cycle $C$ with vertices $v_0, v_1, \ldots , v_m$ and with 
assigned $x$-coordinates $x_1, x_2, \ldots , x_m$, and analogously,
to the cycle $C'$, with vertices $v'_0, v'_1, \ldots , v'_k$ and assigned 
$x$-coordinates $x'_1, x'_2, \ldots , x'_k$. For simplicity, the resulting 
drawings are also denoted by $C$ and $C'$.

Let $x_0$ and $x'_0$ denote the $x$-coordinates of $v_0\in C$ and $v'_0\in C'.$
It follows from (5) of Lemma \ref{slopenumlem22} that $x_0$
is a linear combination of $x_1, x_2, \ldots , x_m$, and $x_0'$
is a linear combination of $x'_1, x'_2, \ldots , x'_k$ with rational coefficients.
Therefore, if $x_0=x'_0$, then there is a nontrivial 
linear combination of $x_1, x_2, \ldots , x_m, x'_1, x'_2, \ldots , x'_k$ 
that gives $0$, contradicting the assumption that these numbers are 
independent over the rationals.  Thus, we can conclude that $x_0\ne x'_0$. 
Assume without loss of generality that $x_0<x'_0$.  
Reflect $C'$ about the $x$-axis, and shift it in the vertical 
direction so that $v'_0$ ends up to the Northeast from $v_0$. Clearly, we can 
add the missing edge $v_0v'_0$. Let $D$ denote the resulting drawing of $G$.  
We claim that $D$ meets all the requirements. Conditions (1), 
(2), (3) and (4) are obviously satisfied, we only have to check that no vertex lies
in the interior of an edge. It follows from Lemma \ref{slopenumlem22} that the $y$-coordinates 
of $v_1, \ldots , v_m$ are all smaller than or equal to the $y$-coordinate of 
$v_0$ and the $y$-coordinates of $v'_1, \ldots , v'_k$ are all greater than or 
equal to the $y$-coordinate of $v'_0$.  We also have $y(v_0)<y(v'_0)$.
Therefore, there is no vertex in the interior of $v_0v'_0$.
Moreover, no edge of $C$ (resp. $C'$) can contain any vertex 
of $v'_0, v'_1, \ldots , v'_k$ (resp. $v_0, v_1, \ldots , v_m$)
in its interior. \hfill $\Box$

\medskip

The rest of the proof is by induction on the number of vertices of $G$. The
statement is trivial if the number of vertices is at most {\em
two}. Suppose that we have already established Theorem \ref{slopenum2} for all
graphs with fewer than $n$ vertices.

Suppose that $G$ has $n$ vertices, it is not a cycle and not
the union of two cycles
connected by one edge.

Unfortunately we have to distinguish several cases. Most of these are very special and less interesting instances that prevent us from using our main argument, which is considered last after clearing all obstacles, as Case 9.
\medskip

\noindent{\bf Case 1:} {\em $G$ has a vertex of degree {\it one}}.
\smallskip

Assume, without loss of generality, that $v_1$ is such a vertex.
If $G$ has no vertex of degree {\em three}, then it consists of a
simple path $P=v_1v_2\ldots v_m$, say. Place $v_m$ at the point
$(x_m,0)$. In general, assuming that $v_{i+1}$ has already been
embedded for some $i<m$, and $x_i<x_{i+1}$, place $v_{i}$ at the
point West of $v_{i+1}$, whose $x$-coordinate is $x_{i}$. If
$x_i>x_{i+1}$, then put $v_{i}$ at the point Northeast of
$v_{i+1}$, whose $x$-coordinate is $x_{i}$. The resulting drawing
of $G=P$ meets all the requirements of the theorem. To see this,
it is sufficient to notice that if $v_j$ would be Northwest of
$v_m$ for some $j<m$, then we could apply Lemma \ref{slopenumlem21} to the cycle
$v_jv_{j+1}\ldots v_m$, and conclude that the numbers $x_j,
x_{j+1},\ldots, x_m$ are dependent over the rationals. This
contradicts our assumption.

Assume next that $v_1$ is of degree {\em one}, and that $G$ has at
least one vertex of degree {\em three}. Suppose without loss of
generality that $v_1v_2\ldots v_kw$ is a path in $G$, whose
internal vertices are of degree {\em two}, but the degree of $w$
is {\em three}. Let $G'$ denote the graph obtained from $G$ by
removing the vertices $v_1, v_2, \ldots , v_{k}$. Obviously, $G'$
is a connected graph, in which the degree of $w$ is {\em two}.

If $G'$ is a {\em cycle}, then apply Lemma \ref{slopenumlem22} to $C=G'$ with $w$
playing the role of the vertex $v_0$ which has no preassigned
$x$-coordinate. We obtain an embedding of $G'$ with edges of
slopes $0, \pi/4,$ and $-\pi/4$ such that $x(v_i)=x_i$ for all
$i>k$ and there is no vertex to the North, to the Northeast, or to
the Northwest of $w$. By (5) of Lemma \ref{slopenumlem22}, $x(w)$ is a linear combination of
$x_{k+1},\ldots, x_m$ with rationals coefficients.
Therefore, $x(w)\neq x_k$, so we can place $v_k$ at the point to
the Northwest or to the Northeast of $w$, whose $x$-coordinate is
$x_k$, depending on whether $x(w)>x_k$ or $x(w)<x_k$. After this,
embed $v_{k-1}, \ldots , v_1$, in this order, so that $v_i$ is
either to the Northeast or to the West of $v_{i+1}$ and
$x(v_i)=x_i$. According to property (4) in Lemma \ref{slopenumlem21}, the path
$v_1v_2\ldots v_k$ lies entirely above $G'$, so no point of
$G$ can lie to the North or to the Northwest of $v_1$.

If $G'$ is {\em not a cycle}, then use the induction hypothesis to
find an embedding of $G'$ that satisfies all conditions of Theorem
\ref{slopenum2}, with $x(w)=x_k$ and $x(v_i)=x_i$ for every $i>k$. Now place
$v_k$ very far from $w$, to the North of it, and draw $v_{k-1},
\ldots , v_1$, in this order, in precisely the same way as in the
previous case. Now if $v_k$ is far enough, then none of the points
$v_k, v_{k-1},\ldots, v_1$ is to the Northwest or to the Northeast
of any vertex of $G'$. It remains to check that condition (4) is true
for $v_1$, but this follows from the fact that there is no point
of $G$ whose $y$-coordinate is larger than that of $v_1$.

\smallskip

{}  From now on, we can and will assume that $G$ has {\em no vertex of
degree one}.

A graph with {\em four} vertices and {\em five} edges between them
is said to be a {\em $\Theta$-graph}.

\medskip
\noindent{\bf Case 2:} {\em $G$ contains a $\Theta$-subgraph.}
\smallskip

Suppose that $G$ has a $\Theta$-subgraph with vertices $a,b,c,d,$
and edges $ab$, $bc$, $ac$, $ad$, $bd$. 
If neither $c$ nor $d$ has a third
neighbor, then $G$ is identical to this graph, which can easily be
drawn in the plane with all conditions of the theorem satisfied.

If $c$ and $d$ are connected by an edge, then all four points of
the $\Theta$-subgraph have degree {\em three}, so that $G$ has no
other vertices. So $G$ is a complete graph of four vertices, and it
has a drawing that meets the
requirements.

Suppose that $c$ and $d$ have a common neighbor $e\neq a,b$. If
$e$ has no further neighbor, then $a,b,c,d,e$ are the only
vertices of $G$, and again we can easily find a proper drawing.
Thus, we can assume that $e$ has a third neighbor $f$. By the
induction hypothesis, $G'=G\setminus \{a,b,c,d,e\}$ has a drawing
satisfying the conditions of Theorem \ref{slopenum2}. In particular, no vertex
of $G'$ is to the North of $f$ (and to the Northwest of $f$,
provided that the degree of $f$ in $G'$ is {\em one}). Further,
consider a drawing $H$ of the subgraph of $G$ induced by the
vertices $a,b,c,d,e$, which satisfies the requirements. We
distinguish two subcases.

If the degree of $f$ in $G'$ is {\em one}, then take a very small
{\em homothetic} copy of $H$ (i.e., similar copy in parallel
position), and rotate it about $e$ in the clockwise direction
through $3\pi/4$. There is no point of this drawing, denoted by
$H'$, to the Southeast of $e$, so that we can translate it into a
position in which $e$ is to the Northwest of $f\in V(G')$ and very
close to it, to a sufficient distance so that (5) is satisfied.
Connecting now $e$ to $f$, we obtain a drawing of $G$
satisfying the conditions. Note that it was important to make $H'$
very small and to place it very close to $f$, to make sure that
none of its vertices is to the North of any vertex of $G'$ whose
degree is at most {\em two}, or to the Northwest of any vertex of
degree {\em one} (other than $f$).

If the degree of $f$ in $G'$ is {\em two}, then we follow the same
procedure, except that now $H'$ is a small copy of $H$, rotated by
$\pi$. We translate $H'$ into a position in which $e$ is to the
North of $f$, and connect $e$ to $f$ by a vertical segment. It is
again clear that the resulting drawing of $G$ meets the
requirements in Theorem \ref{slopenum2}. Thus, we are done if $c$ and $d$ have a
common neighbor $e$.

Suppose now that only one of $c$ and $d$ has a third neighbor,
different from $a$ and $b$. Suppose, without loss of generality,
that this vertex is $c$, so that the degree of $d$ is {\em two}.
Then in $G'=G\setminus \{a,b,d\}$, the degree of $c$ is {\em one}.
Apply the induction hypothesis to $G'$ so that the $x$-coordinate
originally assigned to $d$ is now assigned to $c$ (which had no
preassigned $x$-coordinate in $G$). In the resulting drawing, we
can easily reinsert the remaining vertices, $a, b, d$, by adding a
very small square whose lowest vertex is at $c$ and whose
diagonals are parallel to the coordinate axes. The highest vertex
of this square will represent $d$, and the other two vertices will
represent $a$ and $b$.

We are left with the case when both $c$ and $d$ have a third
neighbor, other than $a$ and $b$, but these neighbors are
different. Denote them by $c'$ and $d'$, respectively. Create a
new graph $G'$ from $G$, by removing $a, b, c, d$ and adding a new
vertex $v$, which is connected to $c'$ and $d'$. Draw $G'$ using
the induction hypothesis, and reinsert $a,b,c,d$ in a small
neighborhood of $v$ so that they form the vertex set of a very
small square with diagonal $ab$. (See Figure \ref{slopenumfig1}.) As before, we
have to choose this square sufficiently small to make sure that
$a, b, c, d$ are not to the North of any vertex $w\neq c',d',v$ of
$G'$, whose degree is at most {\em two}, or to the Northwest of
any vertex of degree {\em one} and pick an appropriate scaling
to make sure that (5) is satisfied.
Thus, we are done if $G$ has a $\Theta$-subgraph.

\smallskip

So, from now on we assume that $G$ has {\em no
$\Theta$-subgraph}.

\begin{figure}[htb]
\epsfxsize=10truecm
\begin{center}
\epsffile{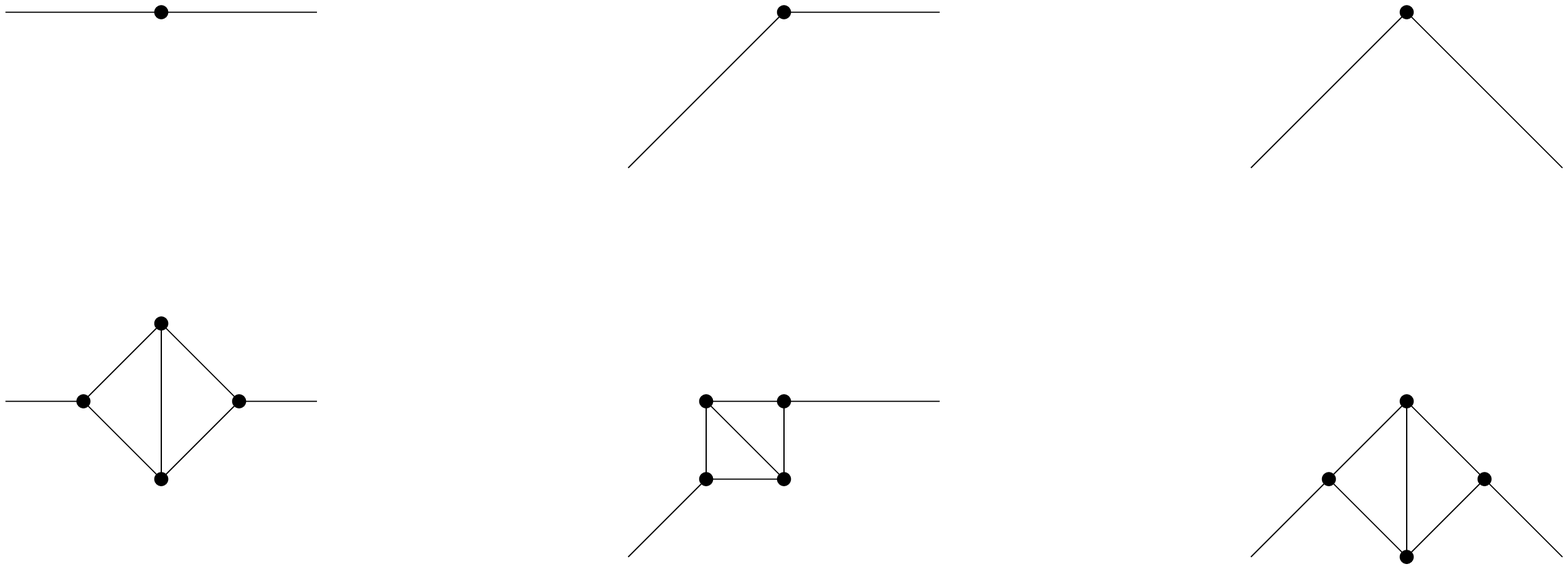}
\caption{Replacing $v$ by $\Theta$.}
\label{slopenumfig1}
\end{center}
\end{figure}

\medskip
\noindent{\bf Case 3:} {\em $G$ has no cycle that passes through a
vertex of degree {\em two}.}
\smallskip

Since $G$ is not three-regular, it contains at least one vertex of
degree {\em two}. Consider a decomposition of $G$ into
$2$-connected blocks and edges. If a block contains a vertex of
degree {\em two}, then it consists of a single edge. The block
decomposition has a treelike structure, so that there is a vertex
$w$ of degree {\em two}, such that $G$ can be obtained as the
union of two graphs, $G_1$ and $G_2$, having only the vertex $w$
in common, and there is no vertex of degree {\em two} in $G_1$.

By the induction hypothesis, for any assignment of rationally
independent $x$-coordinates to all vertices of degree less than
{\em three}, $G_1$ and $G_2$ have proper straight-line embeddings
(drawings) satisfying conditions (1)--(5) of the theorem. The only
vertex of $G_1$ with a preassigned $x$-coordinate is $w$. Applying
a vertical translation, if necessary, we can achieve that in both
drawings $w$ is mapped into the same point. Using the induction
hypothesis, we obtain that in the union of these two drawings,
there is no vertex in $G_1$ or $G_2$ to the North or to the
Northwest of $w$, because the degree of $w$ in $G_1$ and $G_2$ is
{\em one} (property (4)). This is stronger than what we need:
indeed, in $G$ the degree of $w$ is {\em two}, so that we require
only that there is no point of $G$ to the North of $w$ (property
(3)).

The superposition of the drawings of $G_1$ and $G_2$ satisfies all
conditions of the theorem. Only two problems may occur:
\begin{enumerate}
 \item A vertex of $G_1$ may end up at a point to the North
of a vertex of $G_2$ with degree {\em two}.
 \item The (unique) edges in $G_1$ and $G_2$, incident to $w$,
may partially overlap.
\end{enumerate}
Notice that both of these events can be avoided by enlarging the
drawing of $G_1$, if necessary, from the point $w$, and rotating
it about $w$ by $\pi/4$ in the clockwise direction. The latter
operation is needed only if problem 2 occurs. This completes the
induction step in the case when $G$ has no cycle passing through a
vertex of degree {\em two}.

\medskip

\noindent{\bf Case 4:} {\em $G$ has two adjacent vertices of degree {\em two}}.
\smallskip

Take a longest path that contains only degree two vertices. Without loss of generality, assume that this path is $v_1v_2\ldots v_k$. Denote the degree three neighbor of $v_1$ by $u$ and the degree three neighbor of $v_k$ by $w$. 
Let 
$G'=G\setminus \{v_1\ldots v_k\}$. 
Now we distinguish two subcases depending on whether these two vertices are the same or not.

\medskip

{\em Case} 4/a: {\em $u\ne w$}.
\smallskip

First suppose that $G'$ is connected.

If $G'$ is not a cycle, embed it using induction with $x_1$ being the prescribed $x$-coordinate of $u$ and $x_k$ being the prescribed $x$-coordinate of $w$. Now place the $v_i$ vertices one by one high above this drawing, starting with $v_1$, using NW and NE directions. Finally we embed $v_k$ above $w$ and we are done.

If $G'$ is a cycle, then embed it using Lemma \ref{slopenumlem22} with $v_0=u$ and prescribed $x$-coordinate $x_k$ for $w$. Remember that there are no vertices above $u$. So first, we can place $v_1$ to the NW or NE from $u$. Then we place the $v_i$ vertices one by one using NW and NE directions. Finally we embed $v_k$ above $w$ and we are done.

Now suppose 
$G'$ 
has two components. If none of them is a cycle,
embed both of them using induction, high above each other, with $x_1$ being the prescribed $x$-coordinate of $u$ and $x_k$ being the prescribed $x$-coordinate of $w$. Now place the $v_i$ vertices one by one high above the so far drawn components, starting with $v_1$, using NW and NE directions. Finally we embed $v_k$ above $w$ and we are done.

Finally, if $G'=G\setminus \{v_1\ldots v_k\}$ has two components one of which, say the one containing $w$, is a cycle, then embed the component of $u$ using induction with prescribed $x$-coordinate $x_1$ for $u$ or, if it is a cycle, Lemma \ref{slopenumlem22} with $v_0=u$. It is easy to see that we can embed the $v_i$ vertices one by one, starting with $v_1$, just like in the previous cases, and then the rest of the $v_i$ vertices one by one using NW and NE directions. Finally we embed the cycle containing $w$ using Lemma \ref{slopenumlem22} with $v_0=w$, but upside down, so that $w$ has (one of) the smallest $y$-coordinate(s). Shift this cycle vertically such that the edge $v_kw$ has NW or NE direction and we are done.

Note that this last case even works for $k=1$.

\medskip

{\em Case} 4/b: {\em $u=w$}.
\smallskip

Denote the third neighbor of $u$ by $t$.
If the degree of $t$ is two, then deleting the longest path containing $t$ that contains only degree two vertices, the remaining graph will have two components, one of which is a cycle. Thus we end up exactly in the last subcase of Case 4/a, thus we are done.

If the degree of $t$ is three, apply Lemma \ref{slopenumlem22} with $v_0=u$ to the cycle $C=uv_1\ldots v_k$. Denote the $x$-coordinate of $u$ by $x_0$. If $G\setminus C$ is a cycle, we can use Lemma \ref{slopenumlem31}. Otherwise, embed $G\setminus C$ using induction with $x_0$ being the prescribed $x$-coordinate of $t$. Now place $C$ sufficiently high above this drawing.

\medskip

\medskip

\noindent{\bf Case 5 (Main case):} {\em $G$ has a cycle passing through a
vertex of degree {\em two}}.
\smallskip

By assumption, $G$ itself is not a cycle. Therefore, we can also
find a {\em shortest} cycle $C$ whose vertices are denoted by $v,
u_1, \ldots , u_{k}$, in this order, where the degree of $v$ is
{\em two} and the degree of $u_1$ is {\em three}. The length of
$C$ is $k+1$.

It follows from the minimality of $C$ that $u_i$ and $u_j$ are not
connected by an edge of $G$, for any $|i-j|>1$. Moreover, if
$|i-j|>2$, then $u_i$ and $u_j$ do not even have a common neighbor
$(1\le i\neq j\le k)$. This implies that any vertex $v\in
V(G\setminus C)$ has at most {\em three} neighbors on $C$, and
these neighbors must be consecutive on $C$. However, {\em three}
consecutive vertices of $C$, together with their common neighbor,
would form a $\Theta$-subgraph in $G$ (see Case 2). Hence, we can
assume that every vertex belonging to $G\setminus C$ is joined to
at most {\em two} vertices on $C$.


Consider the list $v_1, v_2,\ldots, v_m$ of all vertices of $G$
with degree {\em two}. (Recall that we have already settled the
case when $G$ has a vertex of degree {\em one}.) Assume without
loss of generality that $v_1=v$ and that $v_i$ belongs to $C$ if
and only if $1\le i\le j$ for some $j\le m$.

\medskip
Let ${\bf x}$ denote the {\em assignment} of $x$-coordinates to
the vertices of $G$ with degree {\em two}, that is, ${\bf
x}=(x(v_1), x(v_2), \ldots,$$x(v_m))$$=(x_1, x_2, \ldots, x_m)$.
Given $G$, $C$, ${\bf x}$, and a real parameter $L$, we define the
following so-called {\sc Embedding Procedure$(G, C, {\bf x}, L)$}
to construct a drawing of $G$ that meets all requirements of the
theorem, and satisfies the additional condition that the
$y$-coordinate of every vertex of $C$ is at least $L$ higher than
the $y$-coordinates of all other vertices of $G$.

\medskip
Let $u_1'$ be the neighbor of $u_1$ in $G \setminus C$. 
We mark two different cases here and all the steps in the Embedding
Procedure will be defined for both the cases. 
If $u_1'$ is a vertex of degree three in $G$, we will call it Subcase 5(a), 
and we define $G' = G \setminus C$.
On the other hand, if $u_1'$ is a vertex of degree two, then by Case 4, its
other neighbor (besides $v$), say $u_1''$, is a degree three vertex. We call 
this Subcase 5(b) and define 
$G' = G \setminus (C \cup \{u_1'\})$. 
The main idea
of the Embedding procedure is to inductively embed $G'$ and place the 
rest of the graph in a convenient way. 

\smallskip

Let $B_i$ denote the set of all vertices of $G'$ that
have precisely $i$ neighbors on $C\; (i=0,1,2)$. Thus, we have
$V(G')=B_0\cup B_1\cup B_2$. Further, $B_1=B_1^2\cup
B_1^3$, where an element of $B_1$ belongs to $B_1^2$ or $B_1^3$,
according to whether its degree in $G$ is {\em two} or {\em
three}.

\noindent{\sc Step} 1: If $G'$ is {\em not} a cycle, then 
construct recursively a drawing of
$G'$ satisfying the conditions of Theorem \ref{slopenum2} with the
assignment ${\bf x}'$ of $x$-coordinates $x(v_i)=x_i$ for $j<i\le
m$, and 
$x(u_1')=x_1$ in Subcase 5(a), and, $x(u_1'')= x(u_1')$ in Subcase 5(b).
\smallskip

If $G'$ is a cycle, then, by assumption, there are at least two
edges between $C$ and $G'$. One of them connects $u_1$ to $u_1'$.
Let $u_{\alpha}u'_{\alpha}$ be another such edge, where 
$u_{\alpha}\in C$ and $u'_{\alpha}\in G'$. Since the maximum degree is three, 
$u'_1\ne u'_{\alpha}$. 
Now construct recursively a drawing of
$G'$ satisfying the conditions of Lemma \ref{slopenumlem22}, 
with the exceptional vertex as $u'_{\alpha}$.

\smallskip

We note here that if $G'$ is disconnected, but the components are not
cycles, then we just place them vertically far apart and we still have a 
good recursive drawing of $G'$. 
Suppose that it is disconnected and 
some components are cycles. If the component connected to $u_1$ or
$u_1'$ (based on Subcase 5(a) or 5(b)) is a cycle, 
we draw the cycle exactly as in the
preceding paragraph. For all other components that are cycles, 
we note that since $G$ is connected, there must be at least
one vertex of the cycle connected to $G \setminus G'$ (in fact at least two
because of Case $4$). This is a degree three vertex in $G$ and we will 
call this the exceptional vertex and draw the cycle using Lemma 
\ref{slopenumlem22}. At the end, we shift all components vertically to 
place them sufficiently far apart. We note that this drawing of $G'$
will satisfy all the conditions in Theorem \ref{slopenum2}.

\smallskip

\noindent{\sc Step} 2: For each element of $B_1^2\cup B_2$, take
two rays starting at this vertex, pointing to the Northwest and to
the North. Further, take a vertical ray pointing to the North from
each element of $B_1^3$ and each element of the set $B_{\bf
x}:=\{(x_2, 0), (x_3, 0), \ldots , (x_j, 0)\}$. Let ${\cal R}$
denote the set of all of these rays. Choose the $x$-axis above all
points of $G'$ and all intersection points between the rays in
$\cal R$.

For any $u_h\; (1\le h\le k)$ whose degree in $G$ is {\em three},
define $N(u_h)$ as the unique neighbor of $u_h$ in $G'$.
If $u_h$ has degree {\em two} in $G$, then  $u_h=v_i$ for some
$1\le i\le j$, and let $N(u_h)$ be the point $(x_i, 0)$.
\smallskip

\begin{figure}[htb]
\epsfxsize=6.5truecm
\begin{center}
\epsffile{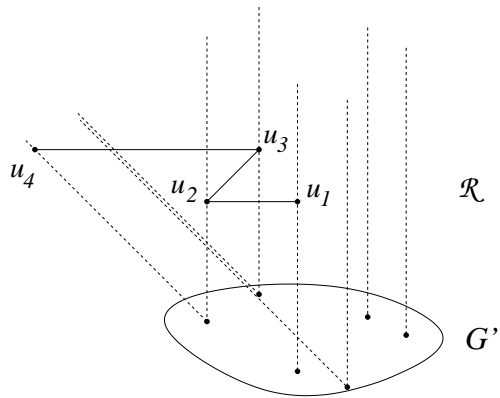}
\caption[Recursively placing vertices]{Recursively place $u_1, u_2, \ldots u_{k}$
on the rays belonging to ${\cal R}$.}
\label{slopenumfig2}
\end{center}
\end{figure}

\noindent{\sc Step} 3: Recursively place $u_1, u_2, \ldots u_{k}$
on the rays belonging to ${\cal R}$, as follows. 
In Subcase 5(a), place $u_1$ on
the vertical ray starting at $N(u_1)=u_1'$ such that $y(u_1)=L$.
In Subcase 5(b), place $u_1'$ on the vertical ray starting at $N(u_1') = u_1''$
such that $y(u_1') = L$. If $x_1< x(u_1')$ then place $u_1$ to the 
West of $u_1'$ on the line $x=x_1$, otherwise place $u_1$ to the 
Northeast of $u_1$, again on $x=x_1$.
Suppose that for some $i<k$ we have already placed $u_1, u_2,
\ldots u_{i}$, so that $L\le y(u_1)\le y(u_2)\le\ldots\le y(u_i)$
and there is no vertex to the West of $u_i$. Next we determine the
place of $u_{i+1}$.

If $N(u_{i+1})\in B_1^2$, then let $r\in{\cal R}$ be the ray
starting at $N(u_{i+1})$ and pointing to the Northwest. If
$N(u_{i+1})\in B_1^3\cup B_{\bf x}$, let $r\in{\cal R}$ be the ray
starting at $N(u_{i+1})$ and pointing to the North. In both cases,
place $u_{i+1}$ on $r$: if $u_i$ lies on the left-hand side of
$r$, then put $u_{i+1}$ to the Northeast of $u_i$; otherwise, put
$u_{i+1}$ to the West of $u_i$.

If $N(u_{i+1})\in B_2$, then let $r\in{\cal R}$ be the ray
starting at $N(u_{i+1})$ and pointing to the North, or, if we have
already placed a point on this ray, let $r$ be the other ray from
$N(u_{i+1})$, pointing to the Northwest, and proceed as before.
\smallskip

\begin{figure}[htb]
\epsfxsize=7truecm
\begin{center}
\epsffile{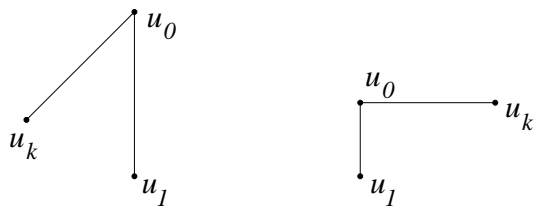}
\caption{Finding the right position for $u_0$.}
\label{slopenumfig3}
\end{center}
\end{figure}

\noindent{\sc Step} 4: Suppose we have already placed $u_{k}$. It
remains to find the right position for $u_0:=v$, which has only
two neighbors, $u_1$ and $u_{k}$. Let $r$ be the ray at $u_{1}$,
pointing to the North. If $u_{k}$ lies on the left-hand side of
$r$, then put $u_0$ on $r$ to the Northeast of $u_k$; otherwise,
put $u_0$ on $r$, to the West of $u_{k}$.

During the whole procedure, we have never placed a vertex on any
edge, and all other conditions of Theorem \ref{slopenum2} are satisfied \hfill $\Box$.
\medskip

Remark that the $y$-coordinates of the vertices $u_0=v, u_1,
\ldots , u_{k}$ are at least $L$ higher than the $y$-coordinates
of all vertices in $G\setminus C$. If we fix $G, C,$ and ${\bf
x}$, and let $L$ tend to infinity, the coordinates of the vertices
given by the above {\sc Embedding Procedure$(G, C, {\bf x}, L)$}
change continuously.

%% file: 1_3_paper_revised.tex
\section{Proof of Theorem \ref{subthm_4slopes}}

\subsection{Assumptions}

This subsection is dedicated to showing that assuming that the cubic
graph is bridgeless and triangle free does not restrict generality. 

We would use theorem \ref{slopenum2} to patch together different components of a cubic graph obtained after removal of some edges. For this we would want to note that we could rotate the components by any multiple of $\pi/4$ and still have a graph with the four basic slopes.

\begin{claim}\label{edge_conn}
A cubic graph with a bridge or a minimal two-edge disconnecting set can be drawn with the four basic slopes. 
\end{claim}

\begin{proof}
We note that the above method cannot be extended to a minimal disconnecting set with more edges, as then, one of the components might be a cycle and then the above theorem cannot be invoked.

Both components obtained by removing the bridge can be drawn with four slopes using Theorem \ref{slopenum2}. Both have the north direction free for the vertex of degree two. To put these together, rotate the second one by $\pi$ and place the degree two vertices above each other. Move the components far enough so that none of the other vertices or edges overlap.

For a two-edge disconnecting set, we may note that these edges must be vertex-disjoint or the graph would contain a bridge. Then, the same procedure as above can be used, now keeping the distance between the two vertices of degree two the same in both components. 
\end{proof}

\begin{claim}\label{vert_conn}
A cubic graph with a cut-vertex or a two-vertex disconnecting set can be drawn with the four basic slopes.
\end{claim}
\begin{proof} 
If the graph has a cut-vertex, then it has a bridge. If it has a two-vertex disconnecting set, then it has a two-edge disconnecting set. In both cases we can then invoke Claim \ref{edge_conn} to draw the graph with four slopes.
\end{proof}

\begin{remark}\label{conn_remark}
A consequence of the above discussion is that any cubic graph that cannot be drawn with the four basic slopes (N,E,NE,NW) must be three vertex and edge connected.
\end{remark}

\begin{claim}\label{triangle_free}
Any cubic graph with a triangle can be drawn with four slopes
\end{claim}
\begin{proof}
First we note that by using the above claims, we may assume that we only consider cubic graphs in which all triangles are connected to the rest of the graph by vertex disjoint edges. If not, then the graph is either $K_4$ or has a two-vertex disconnecting set. A $K_4$ can be drawn using the vertices of a square. In the later case, we can draw the graph with four slopes using Claim \ref{vert_conn}.

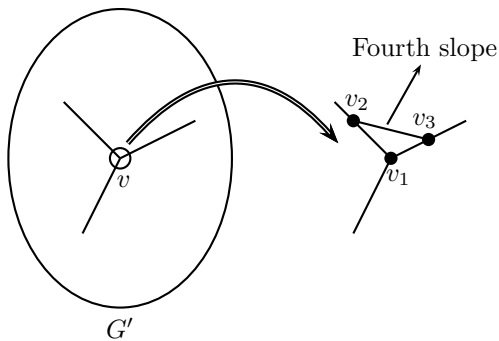
\begin{figure}[h!]
  \begin{center}
    \psset{xunit=.5cm,yunit=.5cm,dotsize=5pt}
    \begin{pspicture}(0,0)(20,8)

	\psellipse[fillcolor=lightgray](4,4)(3,4)
	
	\psline(2.5,5.5)(4,4)
	\psline(3,2)(4,4)
	\psline(6,5)(4,4)

	\psellipse(4,4)(0.3,0.3)

	\parabola[doubleline=true,doublesep=0.015]{->}(4.2,4.4)(7,6)

	\rput(4,-0.5){$G'$}
	\rput(4.1,3.4){$v$}

	\psline(9.7,5.5)(11.2,4)
	\psline(10.2,2)(11.2,4)
	\psline(13.2,5)(11.2,4)
	\psline(10.2,5)(12.2,4.5)

	\psdot(11.2,4)
	\psdot(10.2,5)
	\psdot(12.2,4.5)

	\rput(11.4,3.5){$v_1$}
	\rput(10.3,5.5){$v_2$}
	\rput(12.1,5){$v_3$}

	\psline{->}(11.1,4.9)(12,6.5)
	\rput(12.1,6.9){Fourth slope}

    \end{pspicture}

  \end{center}
  \caption[Adding a triangle]{Adding the triangle to the drawing of $G'$ with four slopes. \label{figure:drawingp}}
\end{figure}

We now prove the claim by contradiction. Suppose there exist cubic graphs with triangles that cannot be drawn with four slopes. By the preceding discussion all triangles in these graphs are necessarily connected to the graphs with vertex-disjoint edges. Of all such graphs consider the one with minimum number of vertices, say $G$. The graph $G'$ obtained by contracting the edges of the triangle $\{v_1,v_2,v_3\}$ is also cubic and has fewer vertices. Either all triangles in $G'$ are connected to the rest of $G'$ with vertex-disjoint edges, in which case we invoke the minimality of $G$ to conclude that $G'$ can be drawn with $4$ slopes (note: here the method of drawing the graph is unknown. We just know there exists a drawing of $G'$ with four slopes). Or, some triangles in $G'$ could be connected to the rest of $G'$ with edges that are not vertex-disjoint. Here we can use Theorem \ref{slopenum2} and the argument of the preceding paragraph to draw $G'$ using four slopes. And lastly, $G'$ could be a triangle-free graph. In this case we use Theorem \ref{subthm_4slopes} to draw $G'$. Hence, $G'$ can always be drawn with four slopes. In $G'$, we call the vertex formed by contracting the edges of the triangle as $v$. Since there is one slope that is not used by the edges incident on $v$, we draw a segment with this slope in a very small neighborhood of $v$ as shown in the figure, to obtaining a drawing of $G$ with four slopes. This contradicts the existence of a minimal counterexample and hence all graphs with triangles can be drawn with four slopes. 

\end{proof}

\begin{remark} \label{triangle_free_remark}
We note here that the preceding Lemma also holds in stricter conditions. To be precise, if the set of basic slopes are sufficient to draw all triangle-free cubic graphs, then they are sufficient to draw all cubic graphs.
\end{remark}

\begin{remark}
It must be noted that this also gives an algorithm for drawing cubic graphs with triangles, namely, we contract triangles until we get a graph that can be drawn with either the Claims \ref{edge_conn},\ref{vert_conn},\ref{triangle_free} or Theorem \ref{slopenum2} or with our drawing strategy for triangle-free bridgeless graphs. Then we can backtrack with placing a series of edges which give us back all the contracted triangles. 
\end{remark}


\subsection{Drawing strategy}

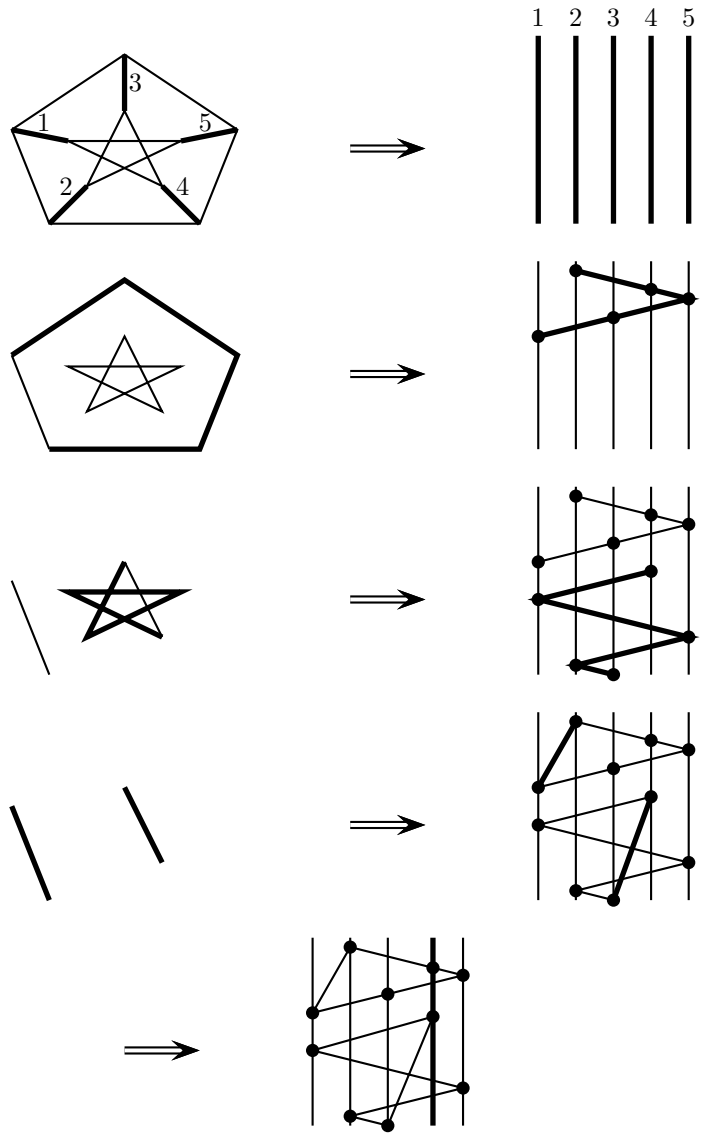
\begin{figure}[h!]
  \begin{center}
    \psset{xunit=.5cm,yunit=.5cm,dotsize=5pt}
    \begin{pspicture}(0,-6)(20,23)

	\pspolygon(0,20.5)(3,22.5)(6,20.5)(5,18)(1,18)
	\pspolygon(1.5,20.2)(4.5,20.2)(2,19)(3,21)(4,19)
	\psline[linewidth=2pt](3,21)(3,22.5)
	\rput(3.29,21.75){3}
	\psline[linewidth=2pt](0,20.5)(1.5,20.2)
	\rput(0.85,20.7){1}
	\psline[linewidth=2pt](1,18)(2,19)
	\rput(1.45,19){2}
	\psline[linewidth=2pt](5,18)(4,19)
	\rput(4.55,19){4}
	\psline[linewidth=2pt](4.5,20.2)(6,20.5)
	\rput(5.15,20.7){5}
	
	\psline[doubleline=true,doublesep=1.5pt]{->}(9,20)(11,20)

	\psline[linewidth=2pt](14,18)(14,23)
	\rput(14,23.5){1}
	\psline[linewidth=2pt](15,18)(15,23)
	\rput(15,23.5){2}
	\psline[linewidth=2pt](16,18)(16,23)
	\rput(16,23.5){3}
	\psline[linewidth=2pt](17,18)(17,23)
	\rput(17,23.5){4}
	\psline[linewidth=2pt](18,18)(18,23)
	\rput(18,23.5){5}


	\psline[linewidth=2pt](0,14.5)(3,16.5)(6,14.5)(5,12)(1,12)
	\psline(0,14.5)(1,12)
	\pspolygon(1.5,14.2)(4.5,14.2)(2,13)(3,15)(4,13)

	\psline[doubleline=true,doublesep=1.5pt]{->}(9,14)(11,14)

	\psline(14,12)(14,17)
	\psline(15,12)(15,17)
	\psline(16,12)(16,17)
	\psline(17,12)(17,17)
	\psline(18,12)(18,17)

	\psline[linewidth=2pt](14,15)(16,15.5)(18,16)(17,16.25)(15,16.75)
	\psdot(14,15)
	\psdot(16,15.5)
	\psdot(18,16)
	\psdot(17,16.25)
	\psdot(15,16.75)


	\psline[linewidth=2pt](4,7)(1.5,8.2)(4.5,8.2)(2,7)(3,9)
	\psline(4,7)(3,9)
	\psline(0,8.5)(1,6)

	\psline[doubleline=true,doublesep=1.5pt]{->}(9,8)(11,8)

	\psline(14,6)(14,11)
	\psline(15,6)(15,11)
	\psline(16,6)(16,11)
	\psline(17,6)(17,11)
	\psline(18,6)(18,11)

	\psline(14,9)(16,9.5)(18,10)(17,10.25)(15,10.75)
	\psline[linewidth=2pt](16,6)(15,6.25)(18,7)(14,8)(17,8.75)
	\psdot(14,9)
	\psdot(16,9.5)
	\psdot(18,10)
	\psdot(17,10.25)
	\psdot(15,10.75)
	\psdot(16,6)
	\psdot(15,6.25)
	\psdot(18,7)
	\psdot(14,8)
	\psdot(17,8.75)


	\psline[linewidth=2pt](4,1)(3,3)
	\psline[linewidth=2pt](0,2.5)(1,0)

	\psline[doubleline=true,doublesep=1.5pt]{->}(9,2)(11,2)

	\psline(14,0)(14,5)
	\psline(15,0)(15,5)
	\psline(16,0)(16,5)
	\psline(17,0)(17,5)
	\psline(18,0)(18,5)

	\psline(14,3)(16,3.5)(18,4)(17,4.25)(15,4.75)
	\psline(16,0)(15,0.25)(18,1)(14,2)(17,2.75)
	\psline[linewidth=2pt](14,3)(15,4.75)
	\psline[linewidth=2pt](16,0)(17,2.75)
	\psdot(14,3)
	\psdot(16,3.5)
	\psdot(18,4)
	\psdot(17,4.25)
	\psdot(15,4.75)
	\psdot(16,0)
	\psdot(15,0.25)
	\psdot(18,1)
	\psdot(14,2)
	\psdot(17,2.75)


	\psline[doubleline=true,doublesep=1.5pt]{->}(3,-4)(5,-4)

	\psline(8,-6)(8,-1)
	\psline(9,-6)(9,-1)
	\psline(10,-6)(10,-1)
	\psline[linewidth=2pt](11.2,-6)(11.2,-1)
	\psline(12,-6)(12,-1)

	\pspolygon(8,-3)(10,-2.5)(12,-2)(11.2,-1.8)(9,-1.25)
	\pspolygon(10,-6)(9,-5.75)(12,-5)(8,-4)(11.2,-3.1)
	\psdot(8,-3)
	\psdot(10,-2.5)
	\psdot(12,-2)
	\psdot(11.2,-1.8)
	\psdot(9,-1.25)
	\psdot(10,-6)
	\psdot(9,-5.75)
	\psdot(12,-5)
	\psdot(8,-4)
	\psdot(11.2,-3.1)

%
%




%
    \end{pspicture}

  \end{center}
  \caption{Process of drawing the cycles. \label{figure:drawingp}}
\end{figure}

Because of the above claims, we would now only focus on graphs that are bridgeless and triangle-free. Since the graph is bridgeless, Petersen's theorem implies that it has a matching. We fix the slope of all the edges in the matching to be $\pi/2$ so that they all lie on (distinct) vertical lines (Figure~\ref{figure:drawingp}). If this matching is removed, then the graph consists of disjoint cycles. Next we isolate one special edge from each cycle. Our method of drawing the graph with four slopes then is as follows: For each cycle, remove the selected edge and draw the remaining path by going between corresponding vertical lines of the cycle alternating with slopes $\pi/4,3\pi/4$ depending on whether we draw the edges with increasing/decreasing $x$-coordinate. This ensures that the cycles all grow upwards. Since we have the freedom to place the cycles where we want, we place them vertically on the matching so that they are very far apart (non-intersecting). Also, if the special edge of each cycle was between adjacent vertical lines then this edge would not pass through any other vertex of the graph either. Then, the only thing we would need is that the final edge in each cycle is drawn with the same slope. Figure~\ref{figure:drawingp} illustrates this and the next remark is followed by a formal description of the problem.

\begin{remark} In \cite{eng} a similar strategy of drawing the matching on vertical lines was employed. However, the cycles were drawn with alternating $\pi/4, 3\pi/4$ slopes for adjacent edges, so that the cycles were not ``growing upwards'' as in our construction. It leads to a different algebraic formulation of the problem giving tight bounds for the case when the cubic graph contains a Hamiltonian cycle.
\end{remark}

Let $M$ be a matching in $G$. Each cycle $C$ in $E(G)\setminus M$
can be represented as a cyclic sequence $C = (v_{1},\ldots,v_{k})$,
where each $v_{i}$ is an element of $M$. The sequence represents the elements of $M$ as we go around the cycle.
We can assume (by Claim \ref{triangle_free}) that
$k\ge 4$. An {\em edge} of $C$ by definition
is $(v_{i},v_{i+1})$ (all indices are understood mod $k$), which is 
although formally 
a pair formed by two distinct elements of $M$, also 
corresponds to an actual edge of the cycle.
Notice that each element of $M$ is either shared by two cycles
or occurs twice in a single cycle.

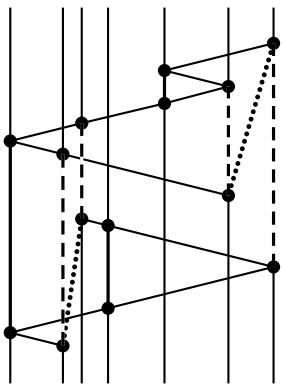
\begin{figure}[h!]
  \begin{center}
    \psset{xunit=.5cm,yunit=.5cm,dotsize=5pt}
    \begin{pspicture}(0,0)(7,10)
	\psline(0,0)(0,10)
	\psline(1.4,0)(1.4,10)
	\psline(1.9,0)(1.9,10)
	\psline(2.6,0)(2.6,10)
	\psline(4.1,0)(4.1,10)
	\psline(5.8,0)(5.8,10)
	\psline(7,0)(7,10)

	\psdot(1.4,1.0)
	\psdot(0,1.35)
	\psdot(2.6,2.0)
	\psdot(7,3.1)
	\psdot(2.6,4.2)
	\psdot(1.9,4.375)

	\psdot(5.8,5)
	\psdot(1.4,6.1)
	\psdot(0,6.45)
	\psdot(1.9,6.925)
	\psdot(4.1,7.45)
	\psdot(5.8,7.9)
	\psdot(4.1,8.325)
	\psdot(7,9.05)

	\psline(1.4,1.0)(0,1.35)(2.6,2.0)(7,3.1)(2.6,4.2)(1.9,4.375)

	\psline(5.8,5)(1.4,6.1)(0,6.45)(1.9,6.925)(4.1,7.45)(5.8,7.9)(4.1,8.325)(7,9.05)

	\psline[linestyle=dotted,dotsep=1pt,linewidth=1.9pt](1.4,1.0)(1.9,4.375)
	\psline[linestyle=dotted,dotsep=1pt,linewidth=1.9pt](5.8,5)(7,9.05)


	\psline[linewidth=1.2pt](0,1.35)(0,6.45)
	\psline[linewidth=1.2pt,linecolor=white](1.4,1.0)(1.4,6.1)
	\psline[linewidth=1.2pt,linestyle=dashed](1.4,1.0)(1.4,6.1)
	\psline[linewidth=1.2pt,linecolor=white](1.9,4.375)(1.9,6.925)
	\psline[linewidth=1.2pt,linestyle=dashed](1.9,4.375)(1.9,6.925)
	\psline[linewidth=1.2pt](2.6,2.0)(2.6,4.2)
	\psline[linewidth=1.2pt](4.1,7.45)(4.1,8.325)
	\psline[linewidth=1.2pt,linecolor=white](5.8,5)(5.8,7.9)
	\psline[linewidth=1.2pt,linestyle=dashed](5.8,5)(5.8,7.9)
	\psline[linewidth=1.2pt,linecolor=white](7,3.1)(7,9.05)
	\psline[linewidth=1.2pt,linestyle=dashed](7,3.1)(7,9.05)

    \end{pspicture}

  \end{center}
  \caption[Distinguished edges and cycles]{Distinguished ``matching-edges'' of Figure 1 are represented by dashed lines while distinguished cycle-edges are represented by dotted lines.\label{figure:distp}}
\end{figure}

We now want to pick a distinguished edge (as in Figure~\ref{figure:distp})
$(v_{i},v_{i+1})$ in $C$ (and in other cycles) such that the 
set of distinguished cycle-edges will satisfy certain
properties. 
\linebreak

Notation: Each distinguished cycle-edge is adjacent with two edges from the matching. These would be called the distinguished matching-edges of the cycle. In particular, the collection of 
distinguished edges from all cycles
form the set of {\em distinguished matching-edges}. 
We would hope that distinguished matching-edges corresponding to a distinguished cycle-edge
can be drawn as adjacent vertical lines for all cycles so that this
would naturally enforce that the distinguished cycle-edge would not go through 
any other vertex of the graph.

\begin{definition}
Two cycles are {\em connected} if 
they share a distinguished matching-edge, and two cycles
{\em belong to the same component} if  they can be reached one from another 
by going through connected cycles. (An alternate way of looking at this would be that two cycles are adjacent iff the sets of distinguished matching-edges corresponding to the two cycles have a non-empty intersection). 
In other words, we define a graph on the cycles
that we call the {\bf cycle-connectivity graph.}
Notice that in this graph each cycle can have at most two neighbors,
thus the graph is a union of paths and cycles.
The set of distinguished matching-edges associated with 
the component where cycle $C$ belongs is denoted by $D(C)$.
(Clearly, if $C_{1}$ and $C_{2}$ belong to the same 
component, then $D(C_{1}) = D(C_{2})$).
\end{definition}

\begin{figure}[h!]
  \begin{center}
    \psset{xunit=.5cm,yunit=.5cm,dotsize=5pt}
    \begin{pspicture}(0,0)(12,12)
	\psline(0,0)(0,12)
	\psline(1,0)(1,12)
	\psline(1.5,0)(1.5,12)
	\psline(2,0)(2,12)
	\psline(3.2,0)(3.2,12)
	\psline(3.6,0)(3.6,12)
	\psline(5.2,0)(5.2,12)
	\psline(6,0)(6,12)

	\psdot(0,1.0)
	\psdot(3.6,1.9)
	\psdot(6,2.5)
	\psdot(1,3.75)

	\psdot(1,5)
	\psdot(0,5.25)
	\psdot(2,5.75)
	\psdot(1.5,5.875)

	\psdot(1.5,7)
	\psdot(3.6,7.525)
	\psdot(3.2,7.625)
	\psdot(2,7.925)

	\psdot(5.2,9.2)
	\psdot(3.2,9.7)
	\psdot(5.2,10.2)
	\psdot(6,10.5)

	\psline(0,1.0)(3.6,1.9)(6,2.5)(1,3.75)
	\rput(4.4,3.3){C1}

	\psline(1,5)(0,5.25)(2,5.75)(1.5,5.875)
	\rput(0.5,5.8){C2}

	\psline(1.5,7)(3.6,7.525)(3.2,7.625)(2,7.925)
	\rput(2.6,8.2){C3}

	\psline(5.2,9.2)(3.2,9.7)(5.2,10.2)(6,10.5)
	\rput(4.4,10.4){C4}

	\psline[linestyle=dotted,dotsep=1pt,linewidth=1.9pt](0,1.0)(1,3.75)
	\psline[linestyle=dotted,dotsep=1pt,linewidth=1.9pt](1,5)(1.5,5.875)
	\psline[linestyle=dotted,dotsep=1pt,linewidth=1.9pt](1.5,7)(2,7.925)
	\psline[linestyle=dotted,dotsep=1pt,linewidth=1.9pt](5.2,9.2)(6,10.5)

	\psline[linewidth=1.2pt](0,1.0)(0,5.25)
	\psline[linewidth=1.2pt](1,3.75)(1,5)
	\psline[linewidth=1.2pt](1.5,5.875)(1.5,7)
	\psline[linewidth=1.2pt](2,5.75)(2,7.925)
	\psline[linewidth=1.2pt](3.2,7.625)(3.2,9.7)
	\psline[linewidth=1.2pt](3.6,1.9)(3.6,7.525)
	\psline[linewidth=1.2pt](5.2,9.2)(5.2,10.2)
	\psline[linewidth=1.2pt](6,2.5)(6,10.5)

	\psdot(9,2)
	\rput(9,2.5){C1}
	\psdot(10,1)
	\rput(10,0.5){C2}
	\psdot(11,2)
	\rput(11,2.5){C3}
	\psdot(12,0.5)
	\rput(12,1.0){C4}

	\psline(9,2)(10,1)(11,2)

    \end{pspicture}

  \end{center}
  \caption{Graph and its connectivity graph.\label{figure:conng}}
\end{figure}
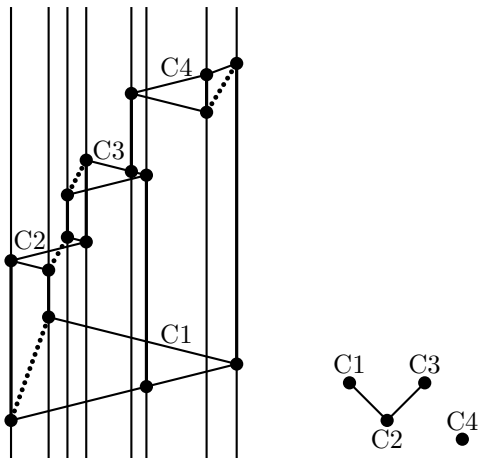

\begin{remark}
We note that in the cycle-connectivity graph two cycles 
are not necessarily connected if they share a matching-edge but only 
if they share a distinguished matching-edge. We can define 
another graph, where two cycles are connected if 
they share any matching-edge. It is easy to see
that $G$ is connected iff the latter graph is connected.
\end{remark}

\begin{remark}
We also note that we may get a multigraph for the cycle-connectivity graph in the event that two cycles pick distinguished cycle-edges between the same set of matching-edges. Condition I below avoids that scenario also. 
\end{remark}

\medskip

\noindent{\bf Condition I:} 
The cycle-connectivity graph does not contain cycles (only paths).
Equivalently,
we can enumerate the distinguished
matching-edges associated with the cycles of a component in
some linear order 
 $y_{1},\ldots, y_{l}$
in such a way that the pairs of consecutive
matching-edges of this order are exactly the distinguished cycle-edges
associated with the cycles in the component.
\medskip

\noindent{\bf Condition II:} 
In each component there is at most one cycle $C$ such that
$C\subseteq D(C)$.

\medskip

Assume that the lines of the matching are ordered $v_1,...,v_n$. From Condition I, we can ensure that every distinguished cycle-edge takes up two adjacent lines in this ordering. A drawing of these lines would be completely determined by the distance between consecutive lines. If $v_i,v_{i+1}$ form a distinguished cycle-edge of the $k^{th}$ cycle, then call the distance between these lines $x_k$. Otherwise fix this distance to be some arbitrary positive constant $c_i$. This is illustrated in Figure~\ref{figure:defxc}.

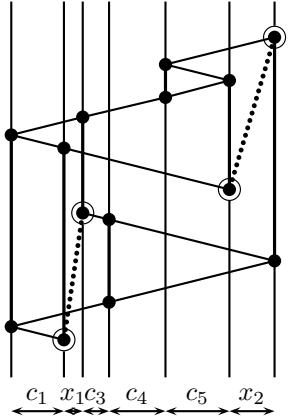
\begin{figure}[h!]
  \begin{center}
    \psset{xunit=.5cm,yunit=.5cm,dotsize=5pt}
    \begin{pspicture}(0,0)(7,10)
	\psline(0,0)(0,10)
	\psline(1.4,0)(1.4,10)
	\psline(1.9,0)(1.9,10)
	\psline(2.6,0)(2.6,10)
	\psline(4.1,0)(4.1,10)
	\psline(5.8,0)(5.8,10)
	\psline(7,0)(7,10)

	\psline{<->}(0,-0.9)(1.4,-0.9)
	\rput(0.7,-0.5){$c_1$}
	\psline{<->}(1.4,-0.9)(1.9,-0.9)
	\rput(1.65,-0.5){$x_1$}
	\psline{<->}(1.9,-0.9)(2.6,-0.9)
	\rput(2.25,-0.5){$c_3$}
	\psline{<->}(2.6,-0.9)(4.1,-0.9)
	\rput(3.35,-0.5){$c_4$}
	\psline{<->}(4.1,-0.9)(5.8,-0.9)
	\rput(4.95,-0.5){$c_5$}
	\psline{<->}(5.8,-0.9)(7,-0.9)
	\rput(6.4,-0.5){$x_2$}

	\psdot[dotstyle=o,dotscale=1.75](1.4,1.0)
	\psdot(1.4,1.0)
	\psdot(0,1.35)
	\psdot(2.6,2.0)
	\psdot(7,3.1)
	\psdot(2.6,4.2)
	\psdot[dotstyle=o,dotscale=1.75](1.9,4.375)
	\psdot(1.9,4.375)

	\psdot[dotstyle=o,dotscale=1.75](5.8,5)
	\psdot(5.8,5)
	\psdot(1.4,6.1)
	\psdot(0,6.45)
	\psdot(1.9,6.925)
	\psdot(4.1,7.45)
	\psdot(5.8,7.9)
	\psdot(4.1,8.325)
	\psdot[dotstyle=o,dotscale=1.75](7,9.05)
	\psdot(7,9.05)

	\psline(1.4,1.0)(0,1.35)(2.6,2.0)(7,3.1)(2.6,4.2)(1.9,4.375)

	\psline(5.8,5)(1.4,6.1)(0,6.45)(1.9,6.925)(4.1,7.45)(5.8,7.9)(4.1,8.325)(7,9.05)

	\psline[linestyle=dotted,dotsep=1pt,linewidth=1.9pt](1.4,1.0)(1.9,4.375)
	\psline[linestyle=dotted,dotsep=1pt,linewidth=1.9pt](5.8,5)(7,9.05)

	\psline[linewidth=1.2pt](0,1.35)(0,6.45)
	\psline[linewidth=1.2pt](1.4,1.0)(1.4,6.1)
	\psline[linewidth=1.2pt](1.9,4.375)(1.9,6.925)
	\psline[linewidth=1.2pt](2.6,2.0)(2.6,4.2)
	\psline[linewidth=1.2pt](4.1,7.45)(4.1,8.325)
	\psline[linewidth=1.2pt](5.8,5)(5.8,7.9)
	\psline[linewidth=1.2pt](7,3.1)(7,9.05)

    \end{pspicture}

  \end{center}
  \caption{Definition of variables $x_i$ and $c_i$. \label{figure:defxc}}
\end{figure}

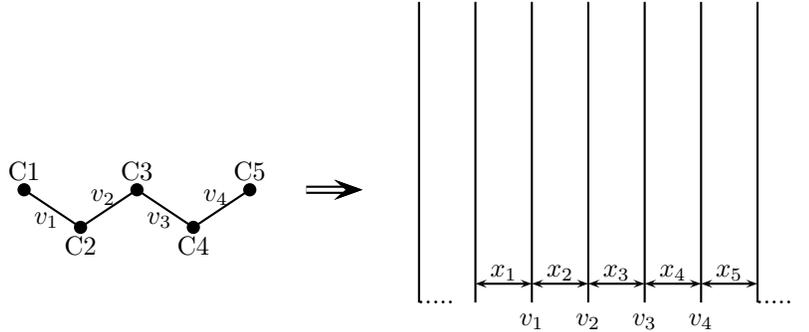
\begin{figure}[h!]
  \begin{center}
    \psset{xunit=0.75cm,yunit=.5cm,dotsize=5pt}
    \begin{pspicture}(0,0)(14,10)
	\psdot(0,5)
	\rput(0,5.5){C1}
	\psdot(1,4)
	\rput(1,3.5){C2}
	\psdot(2,5)
	\rput(2,5.5){C3}
	\psdot(3,4)
	\rput(3,3.5){C4}
	\psdot(4,5)
	\rput(4,5.5){C5}

	\psline(0,5)(1,4)(2,5)(3,4)(4,5)
	\rput(0.4,4.2){$v_1$}
	\rput(1.4,4.8){$v_2$}
	\rput(2.4,4.2){$v_3$}
	\rput(3.4,4.8){$v_4$}

	\psline[doubleline=true,doublesep=1.5pt]{->}(5,5)(6,5)

	\psline(7,2)(7,10)
	\rput(7.3,2){.....}
	\psline(8,2)(8,10)
	\psline(9,2)(9,10)
	\rput(9,1.5){$v_1$}
	\psline(10,2)(10,10)
	\rput(10,1.5){$v_2$}
	\psline(11,2)(11,10)
	\rput(11,1.5){$v_3$}
	\psline(12,2)(12,10)
	\rput(12,1.5){$v_4$}
	\psline(13,2)(13,10)
	\rput(13.3,2){.....}
	\psline(14,2)(14,10)

	\psline{<->}(8,2.5)(9,2.5)
	\rput(8.5,2.8){$x_1$}
	\psline{<->}(9,2.5)(10,2.5)
	\rput(9.5,2.8){$x_2$}
	\psline{<->}(10,2.5)(11,2.5)
	\rput(10.5,2.8){$x_3$}
	\psline{<->}(11,2.5)(12,2.5)
	\rput(11.5,2.8){$x_4$}
	\psline{<->}(12,2.5)(13,2.5)
	\rput(12.5,2.8){$x_5$}
`

    \end{pspicture}

  \end{center}
  \caption[Cycle connectivity graph]{Paths of cycles will have adjacent distinguished cycle-edges in the drawing (because of the distinguished matching-edge they share). Hence it is necessary to not have cycles in the connectivity graph. \label{figure:condI}}
\end{figure}

Now draw a cycle by starting at one of the distinguished matching-edges and first drawing the path obtained by removing the distinguished cycle-edge. If an edge of the cycle is $v_k,v_l$ where $k<l$ then use a slope of $\pi/4$ and $3\pi/4$ otherwise. Notice that the vertical distance traveled across this edge is equal to the distance between the lines $v_k$ and $v_l$. Hence the slope of the distinguished cycle-edge would look like 
$g_{i} = {\LL_{i}(\x)
\over x_{i}}$ 
where $\LL_{i}(\x)=a_{i,0} + \sum_{j=1}^{n} a_{i,j}x_{j}$ for $1\le i \le m$ ($m$ being the number of cycles)
is a linear equation on $\x$ with non-negative coefficients. We will use the following Solvability Theorem to ensure that these slopes can always be matched. This will be proved in the next subsection. 

\begin{theorem}\label{generalt}
Let
$\LL_{i}(\x)=a_{i,0} + \sum_{j=1}^{n} a_{i,j}x_{j}$ for $1\le i \le n$ 
be linear forms, such that 
all coefficients are non-negative. 
Define a directed graph, $\G = \G(\overline{\LL})$ with 
vertex set $V(\G) = \{0,1,\ldots, n\}$ and edge
set $E(\G) = \{(j,i)\;\mid\; a_{i,j} \neq 0\}$.
Let $g_{i} = {\LL_{i}(\x)
\over x_{i}}$ for $1\le i \le n$. Assume that 
in $\G(\overline{\LL})$ every node can be reached from $0$. Then
\begin{equation}\label{maineq2}
g_{1}(\x) = g_{2}(\x) = \cdots = g_{n}(\x)
\end{equation}
has an all-positive solution.
\end{theorem}

\begin{definition}
We define $r(i) = dist(0,i)$ in the above graph $\G(\overline{\LL})$ and for a cycle $C$ if the variable was $x_i$ for its distinguished cycle-edge, we would denote $r(C)$ to mean $r(i)$.
\end{definition}

\begin{theorem}
If Conditions I and II hold then we can use Theorem \ref{generalt}
to prove that every connected graph $G$ is implementable with four directions.
\end{theorem}
\begin{proof}
Condition I
ensures that the slope associated with the
distinguished cycle-edge of each cycle $i$ can be expressed as $g_{i}(\x)$
(as we have seen).
Condition II is sufficient for the reachability condition (for $\G$) of 
Theorem \ref{generalt}. 
We will in fact show that $r(C)\le 2$ for every cycle $C$.
The linear expression
for cycle $C$ has a non-zero constant term iff 
$C\setminus D(C)\neq \emptyset$.
Consider a fixed component. 
By Condition II all cycles, except perhaps one, 
have associated linear expressions
with non-zero constant terms, therefore they have $r=1$.

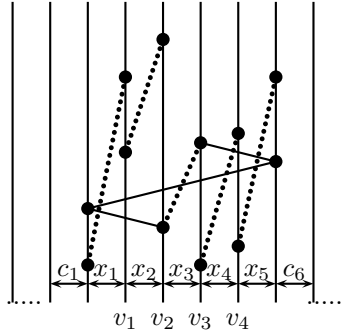
\begin{figure}[h!]
  \begin{center}
    \psset{xunit=.5cm,yunit=.5cm,dotsize=5pt}
    \begin{pspicture}(6,0)(15,10)
	\psline(6,2)(6,10)
	\rput(6.3,2){.....}
	\psline(7,2)(7,10)
	\psline(8,2)(8,10)
	\psline(9,2)(9,10)
	\rput(9,1.5){$v_1$}
	\psline(10,2)(10,10)
	\rput(10,1.5){$v_2$}
	\psline(11,2)(11,10)
	\rput(11,1.5){$v_3$}
	\psline(12,2)(12,10)
	\rput(12,1.5){$v_4$}
	\psline(13,2)(13,10)
	\psline(14,2)(14,10)
	\rput(14.3,2){.....}
	\psline(15,2)(15,10)

	\psline{<->}(7,2.5)(8,2.5)
	\rput(7.5,2.8){$c_1$}
	\psline{<->}(8,2.5)(9,2.5)
	\rput(8.5,2.8){$x_1$}
	\psline{<->}(9,2.5)(10,2.5)
	\rput(9.5,2.8){$x_2$}
	\psline{<->}(10,2.5)(11,2.5)
	\rput(10.5,2.8){$x_3$}
	\psline{<->}(11,2.5)(12,2.5)
	\rput(11.5,2.8){$x_4$}
	\psline{<->}(12,2.5)(13,2.5)
	\rput(12.5,2.8){$x_5$}
	\psline{<->}(13,2.5)(14,2.5)
	\rput(13.5,2.8){$c_6$}

	\psdot(10,4)
	\psdot(8,4.5)
	\psdot(13,5.75)
	\psdot(11,6.25)

	\psline(10,4)(8,4.5)(13,5.75)(11,6.25)
	\psline[linestyle=dotted,dotsep=1pt,linewidth=1.9pt](10,4)(11,6.25)

	\psline[linestyle=dotted,dotsep=1pt,linewidth=1.9pt](8,3)(9,8)
	\psline[linestyle=dotted,dotsep=1pt,linewidth=1.9pt](9,6)(10,9)
	\psline[linestyle=dotted,dotsep=1pt,linewidth=1.9pt](11,3)(12,6.5)
	\psline[linestyle=dotted,dotsep=1pt,linewidth=1.9pt](12,3.5)(13,8)

	\psdot(8,3)
	\psdot(9,8)
	\psdot(9,6)
	\psdot(10,9)
	\psdot(11,3)
	\psdot(12,6.5)
	\psdot(12,3.5)
	\psdot(13,8)

    \end{pspicture}

  \end{center}
  \caption[Cycles spanning over only distinguished edges]{Here the dotted edges represent a set of adjacent distinguished cycle-edges. $r(C) \neq 1$ if all edges of the cycle span over these adjacent distinguished cycle-edges. But all $v_i$'s in the figure have both vertices of the matching-edges used up by cycles. So $C$ could at best be a 4 cycle (since the graph is triangle-free) using up the first and the last vertical lines of this contiguous block and one distinguished cycle-edge. \label{figure:spcaserC}}
\end{figure}

It is sufficient to show that 
the single cycle $C$ for which
$C\subseteq D(C)$, if exists, has $r(C)=2$. Indeed,
let $y_{1},\ldots, y_{l}$ be the distinguished matching-edges belonging to 
this component in this linear order,
and let $y_{p}$ and $y_{p+1}$ be the distinguished matching-edges that belong to 
cycle $C$. Since $C$ is at least a four cycle, it either
contains some other $y_{p'}\not\in\{y_{p},y_{p+1}\}$,
in which case indeed, it is geometrically easy to see that 
one of the other variables from the component has to occur in
$\LL_{C}$ or $C$ is a four cycle and both $y_{p}$ and $y_{p+1}$
occur with multiplicity two in it. In the latter case 
$C$ would form a separate $K_{4}$ component, thus $G=K_{4}$.
In the former case the variable has $r=1$, so $r(C)=2$.
\end{proof}

We are left with proving that we can pick 
distinguished cycle-edges from the cycles such that 
Conditions I and II are satisfied.
Indeed, start from any cycle, and pick an
edge for a distinguished cycle-edge, which has at least 
one adjacent matching-edge $y$ that is common with a different cycle.
If there is none, the cycle is the single (Hamiltonian) cycle,
and if we distinguish any edge, Conditions I and II are clearly satisfied.
Otherwise, in the cycle that contains $y$, pick one of the two edges 
adjacent to $y$, look at the other adjacent matching-edge, $y'$, of this edge,
look for another cycle that is adjacent with $y'$, etc.
The process ends when we get back to any cycle (including the current one)
that has already been visited.
There is one reason for back-track and this is when we return to the 
other adjacent matching-edge, $z$, of the starting edge. In this case we choose
the other edge (recall we always have two choices).
It would be fatal to get back to $z$, since then Condition
I would not hold.

Assume that the above procedure has gone through. Then we have 
distinguished at most three matching-edges adjacent to any cycle. But this is not all. 
We have to do the same procedure from $z$ as well. 
The procedure terminates when we encounter a cycle that has already 
been encountered. Thus in the final step we might create a fourth 
distinguished matching-edge adjacent to one of the cycles, but only in one of them.
This can be the single cycle $C$ in
the component for which $C\subseteq D(C)$. And because the graph is triangle-free, all the other components would have $C \setminus D(C) \neq \emptyset$.

Once we are done with creating
the first component, we select a cycle not
involved in it, and start the same procedure as before
with the only difference that in 
subsequent rounds we also stop if we encounter 
a cycle visited in one of the previous rounds. It is easy to see,
that now for the distinguished cycle-edges that we have selected Conditions I and II hold.

\subsection{Solvability}

Before we prove Theorem \ref{generalt}, we will look at the following special case when all the constant terms in $\LL_i$ are positive.

\begin{theorem}\label{specialth}
Let $B_{1},\ldots,B_{n} > 0$ be positive constants, 
$\LL_{i}(\x)=\sum_{j=1}^{n} a_{i,j}x_{j}$ for $1\le i \le n$ be linear forms. 
Let $g_{i} = {B_{i} + \LL_{i}(\x)
\over x_{i}}$ for $1\le i \le n$. Then
\begin{equation}\label{maineq}
g_{1}(\x) = g_{2}(\x) = \cdots = g_{n}(\x)
\end{equation}
has an all-positive solution.
\end{theorem}

\begin{proof} The intuition behind the proof is this: Let $\epsilon$ be very 
small and $\alpha_{1},\ldots,\alpha_{n}> 0$ be fixed. If we set
$x_{i} = \epsilon B_{i}\alpha_{i}^{-1}$ then 
$g_{i}(\x)\approx \epsilon^{-1}\alpha_{i}$. In particular, let $\Alpha$ range
in the $[1,2]^{n}$ solid cube. Then, if $\epsilon$ is small enough, the 
vector $(g_{1}(\x),\ldots,g_{n}(\x))$ will range roughly in the 
$[\epsilon^{-1},2\epsilon^{-1}]^{n}$ cube, thus 
$\epsilon^{-1}(1.5,\ldots,1.5)$,
which is the center of this cube, has to be in the image.

To make this proof idea precise we will use the following version of
Brouwer's well known fix point theorem:

\begin{theorem}[Brouwer]\label{brower}
Let $f:[1,2]^{n}\rightarrow [1,2]^{n}$ be a continuous function. Then 
$f$ has a fix point, i.e. an $\x_{0} \in [1,2]^{n}$ for which
$f(\x_{0}) = \x_{0}$.
\end{theorem}
We will use the fix point theorem as below.
We first define 
\[
h(\alpha_{1},\ldots,\alpha_{n}) = (\epsilon g_{1}(\x),
\ldots,\epsilon g_{n}(\x)),
\]
where $\x = \epsilon (\alpha_{1}^{-1} B_{1},\ldots,
\alpha_{n}^{-1} B_{n}) = \epsilon \x'$, 
and we think of $\epsilon$ as some fixed 
positive number. Notice that $\x'$ is just a function 
of $\Alpha$, independent of $\epsilon$.
It is sufficient to show that 
if $\epsilon$ is small enough,
there are $\alpha_{1},\ldots,\alpha_{n}$  such that $h(\Alpha) = 
(1.5,\ldots,1.5)$, since then $\x$ satisfies
(\ref{maineq}) with common value $1.5 \epsilon^{-1}$. We have:
\[
\epsilon g_{i}(\x) =
\epsilon {B_{i} + \LL_{i}(\x) \over 
\epsilon \alpha_{i}^{-1} B_{i}} = \alpha_{i}( 1 + \epsilon B_{i}^{-1}
\LL_{i}(\x')).
\]
Here we used that $\LL_{i}(\epsilon \x') = \epsilon \LL_{i}(\x')$. 
We would like to have
\begin{equation}\label{secondeq}
\alpha_{i}( 1 + \epsilon B_{i}^{-1}
\LL_{i}(\x')) = 1.5 \;\;\;\;\; \mbox{for $1\le i\le n$}.
\end{equation}
Define
\begin{eqnarray*}
K & = & \max_{i} \sup_{\Alpha \in [1,2]^{n}} 
B_{i}^{-1} \LL_{i}(\x'); \\
\epsilon & = & 1/(10 K).
\end{eqnarray*}
To use the fix point theorem we consider the map
\[
f: (\alpha_{1},\ldots,\alpha_{n}) \rightarrow 
\left(
{1.5 \over
 1 + \epsilon B_{1}^{-1}
\LL_{1}(\x')}, 
\ldots , 
{1.5 \over
 1 + \epsilon B_{n}^{-1}
\LL_{n}(\x')}\right)
\]
on the cube $[1,2]^{n}$. The image is contained in
$[1,2]^{n}$, since if $\Alpha\in [1,2]^{n}$ then 
for $1\le i \le n$ we have 
\[
1\; <\; {1.5 \over 1+ 0.1} \; = \; {1.5 \over 1+ \epsilon K} \; \le\;
{1.5 \over
 1 + \epsilon B_{i}^{-1}
\LL_{i}(\x')} \; \le \; {1.5 \over 1- \epsilon K}  \; = \;
 {1.5 \over 1 - 0.1}  \; < \; 2.
\]
Therefore, by Theorem \ref{brower} there is an $\Alpha\in [1,2]^{n}$ such that
$\alpha_{i} = {1.5 \over
 1 + \epsilon B_{i}^{-1}
\LL_{i}(\x')}$ for $1\le i \le n$, which is equivalent to
(\ref{secondeq}).
\end{proof}

In Theorem \ref{specialth} all linear forms have non-zero constant terms. 
We can, however generalize this to Theorem \ref{generalt}. We discuss its proof below.

\begin{remark}
The non-negativity of the coefficients can be relaxed such that
the theorem becomes a true generalization of Theorem \ref{specialth}.
Since the more general condition is slightly technical,
we will stay with the simpler non-negativity condition, which is 
sufficient for us.
\end{remark}

\begin{proof}
For $1\le i\le n$ let $r(i)= dist(0,i)$ in $\G(\overline{\LL})$.
(In Theorem \ref{specialth} each $r(i)$ was $1$.) Define
\[
x_{i} = \epsilon^{r(i)}x'_{i},
\]
where $\epsilon>0$ will be a small enough
number that we will appropriately fix later, but as of now 
we think about it as a quantity tending to zero.
We can rewrite (\ref{maineq}) as:
\[
\epsilon g_{1}(\x) = \epsilon g_{2}(\x) = \cdots = \epsilon g_{n}(\x).
\]
If we fix $\x'$ and take epsilon tending to zero, then,
\[
\epsilon g_{i}(\x)\rightarrow {\beta_{i}(\x')\over x'_{i}},
\]
where $\beta_{i}(\x') = a_{i,0}/x'_{i}$ if $r(i) = 1$, otherwise
\[
\beta_{i}(\x') = \sum_{j:\;r(j) = r(i)-1} a_{i,j} x'_{j}.
\]
We can now solve the system
\[
{\beta_{i}(\x')\over x'_{i}} = 1.5
\]
and even the system
\begin{equation}\label{syseq}
{\beta_{i}(\x')\over x'_{i}} = \alpha_{i},
\end{equation}
where $1 \le \alpha_{i} \le 2$ for $1\le i\le n$.
Indeed, the solution can be obtained iteratively, by
first computing the values of the variables 
$x_i$ with $r(i)=0$, then with $r(i)=1$, etc.
We can again use the fix point theorem of Brouwer to show that
if $\epsilon$ is sufficiently small, the system
\[
\epsilon g_{i}(\x) = 
1.5\hspace{0.5in}\mbox{for $1\le i\le n$}
\]
has a solution. For this we again parameterize $\x'$ with
$\Alpha$. When $\Alpha$ ranges in
the solid cube $[1,2]^{n}$ then $\x'$ will range in 
some domain $D$, where we obtain $D$ by solving the system
(\ref{syseq}) for all $\alpha_{i}\in [1,2]^{n}$.
Now we have to set $\epsilon$ small enough such that everywhere
in $D$ it should hold that
\begin{equation}\label{rgeq}
0.9 \le 
{ {\beta_{i}(\x') / x'_{i}} \over \epsilon g_{i}(\x)}
= {\alpha_{i} \over \epsilon g_{i}(\x)}
\le 1.1 \hspace{0.5in}\mbox{for $1\le i\le n$.}
\end{equation}
This is easily seen to be possible, since $D$ is contained in
a closed cube in the strictly positive orthant.
We then apply the fix point theorem to 
\[
f: \Alpha \rightarrow \boldsymbol{\gamma},
\]
where
\[
\gamma_{i} = { 1.5\alpha_{i} 
\over
\epsilon g_{i}(\x)}.
\]
The fix point theorem applies, since the
range of $f$ remains in the $[0.9\cdot 1.5,1.1\cdot 1.5]^{n}\subset
[1,2]^{n}$
cube by Equation (\ref{rgeq}).
For the fixed point $\alpha_{i} =  { 1.5\alpha_{i} 
\over
\epsilon g_{i}(\x)}$ for $1\le i\le n$, which implies
$\epsilon g_{i}(\x)=1.5$ for $1\le i\le n$. 
\end{proof}

%% file: 1_4_slope_new_final.tex
\section{Proof of Theorem \ref{thm_4basic_slopes}}

We start with some definitions we will use throughout this section.

\subsection{Definitions}
\indent \vspace{-0.3cm}

Throughout this section $\log$ always denotes $\log_2$, the logarithm in base $2$.

We recall that the girth of a graph is the length of its shortest cycle.

\begin{definition}
Define a {\em  supercycle} as a connected graph where every degree is at least two and not all are two. Note that a minimal supercycle will look like a ``$\theta$'' or like a ``dumbbell''. 
\end{definition}

We recall that a {\em cut} is a partition of the vertices into two sets. We say that an edge is in the cut if its ends are in different subsets of the partition. We also call the edges in the cut the {\em cut-edges}. The {\em size} of a cut is the number of cut-edges in it.

\begin{definition}
We say that a cut is an {\em $M$-cut} if the cut-edges form a matching, in other words, if their ends are pairwise different vertices.
We also say that an {\em $M$-cut} is suitable if after deleting the cut-edges, the graph has two components, both of which are supercycles.
\end{definition}


We refer the reader to Section \ref{slope_intro} for the exact statement of Theorem \ref{slopenum2} \cite{kppt08_2} about subcubic graphs.

Note that Theorem \ref{slopenum2} proves the result of Theorem \ref{thm_4basic_slopes} for subcubic graphs. 
Another minor observation is that we may assume that the graph is connected. 
Since we use the basic four slopes, if we can draw the components of a 
disconnected graph, then we just place them far apart in the plane so 
that no two drawings intersect. 
So we will assume for the rest of the section that the graph is cubic and connected.

\subsection{Preliminaries}
The results in this subsection are also interesting independent of the current problem we deal with. The following is also called the Moore bound.

\begin{lemma}\label{girth}
Every connected cubic graph on $n$ vertices contains a cycle of length at most 
$2 \lceil \log ( \frac{n}{3} +1) \rceil $.
\end{lemma}

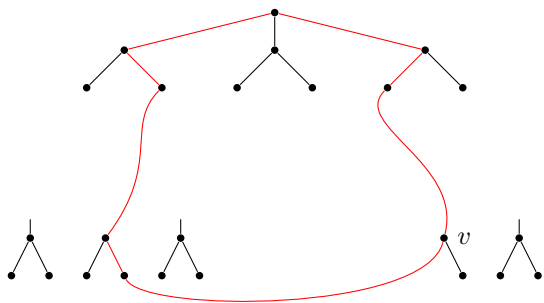
\begin{figure*}[h]
{\centering
\begin{tikzpicture}[scale=0.5]
\node [fill=black,circle,inner sep=1pt] (1) at (7,7) {}; 
\node [fill=black,circle,inner sep=1pt] (2) at (3,6) {}; 
\node [fill=black,circle,inner sep=1pt] (3) at (7,6) {}; 
\node [fill=black,circle,inner sep=1pt] (4) at (11,6) {}; 
\node [fill=black,circle,inner sep=1pt] (5) at (2,5) {}; 
\node [fill=black,circle,inner sep=1pt] (6) at (4,5) {}; 
\node [fill=black,circle,inner sep=1pt] (7) at (6,5) {}; 
\node [fill=black,circle,inner sep=1pt] (8) at (8,5) {}; 
\node [fill=black,circle,inner sep=1pt] (9) at (10,5) {}; 
\node [fill=black,circle,inner sep=1pt] (10) at (12,5) {}; 
\node [fill=black,circle,inner sep=1pt] (11) at (0,0) {}; 
\node [fill=black,circle,inner sep=1pt] (12) at (0.5,1) {}; 
\node [fill=black,circle,inner sep=1pt] (13) at (1,0) {}; 
\node [fill=black,circle,inner sep=1pt] (14) at (2,0) {}; 
\node [fill=black,circle,inner sep=1pt] (15) at (2.5,1) {}; 
\node [fill=black,circle,inner sep=1pt] (16) at (3,0) {}; 
\node [fill=black,circle,inner sep=1pt] (17) at (4,0) {}; 
\node [fill=black,circle,inner sep=1pt] (18) at (4.5,1) {}; 
\node [fill=black,circle,inner sep=1pt] (19) at (5,0) {}; 
\node (20) at (11,-1) {}; 
\node (26) at (3.5,-1) {}; 
\node [fill=black,circle,inner sep=1pt,label=right:$v$] (21) at (11.5,1) {}; 
\node [fill=black,circle,inner sep=1pt] (22) at (12,0) {}; 
\node [fill=black,circle,inner sep=1pt] (23) at (13,0) {}; 
\node [fill=black,circle,inner sep=1pt] (24) at (13.5,1) {}; 
\node [fill=black,circle,inner sep=1pt] (25) at (14,0) {}; 
\draw [black] (2) -- (5);
\draw [black] (4) -- (10);
\draw [red] (2) -- (1) -- (4);
\draw [black] (1) -- (3) -- (7);
\draw [red] (2) -- (6);
\draw [black] (8) -- (3);
\draw [red] (9) -- (4);

\draw [black] (11) -- (12) -- (13);
\draw [black] (12) -- (0.5,1.5);
\draw [black] (18) -- (4.5,1.5);
\draw [black] (24) -- (13.5,1.5);
\draw [black] (14) -- (15);
\draw [red] (16) -- (15);
\draw [black] (17) -- (18) -- (19);
\draw [black] (23) -- (24) -- (25);
\draw [black] (21) -- (22);
\draw [red] (21) .. controls (20) and (26) .. (16);
\draw [red] (6) .. controls (3,4) and (4,3) .. (15);
\draw [red] (9) .. controls (9,4) and (12,3) .. (21);

\end{tikzpicture}
\caption[Cycle in BFS tree]{Finding a cycle in the BFS tree using that the left child of $v$ already occurred. 
}
} 
\label{fig:cycle}
\end{figure*}

\begin{proof}
Start at any vertex of $G$ and conduct a breadth first search (BFS) of 
$G$ until a vertex repeats in the BFS tree. 
We note here that by iterations we will (for the rest of the subsection)
mean the number of levels of the BFS tree.
Since $G$ is cubic, after $k$ iterations, the number of vertices visited
will be $1 + 3+ 6+ 12 + \ldots +3 \cdot 2^{k-2} = 1+3(2^{k-1}-1)$. 
And since $G$ has $n$ vertices, some vertex must repeat after 
$k = \lceil \log (\frac{n}{3} + 1) \rceil +1$ iterations.
Tracing back along the two paths obtained for the vertex that reoccurs,
we find a cycle of length at most
$2 \lceil \log ( \frac{n}{3} +1) \rceil $.
\end{proof}

\begin{lemma}\label{supercycle}
Every connected cubic graph on $n$ vertices with girth $g$ contains a supercycle with at most $2 \lceil \log (\frac{n+1}{g}) \rceil +g-1$ vertices.
\end{lemma}
\begin{proof}
Contract the vertices of a length $g$ cycle, obtaining a multigraph $G'$ with $n-g+1$ vertices, that is almost $3$-regular, except for one vertex of degree $g$, from which we start a BFS. It is easy to see that the number of vertices visited after $k$ iterations is at most $1 + g + 2g + 4g + \ldots + g\cdot 2^{k-2} = g(2^{k-1} -1)+1$. 
And since $G'$ has $n-g+1$ vertices, some vertex must repeat after 
$k = \lceil \log (\frac{n-g+1}{g} + 1) \rceil +1=\lceil \log (\frac{n+1}{g}) \rceil +1$ iterations.
Tracing back along the two paths obtained for the vertex that reoccurs,
we find a cycle (or two vertices connected by two edges) of length at most $2 \lceil \log (\frac{n+1}{g}) \rceil$ in $G'$.
This implies that in $G$ we have a supercycle with at most $2 \lceil \log (\frac{n+1}{g}) \rceil +g-1$ vertices.
\end{proof}

\begin{lemma}\label{Mcut}
Every connected cubic graph on $n>2s -2$ vertices with a supercycle with $s$ vertices contains a suitable $M$-cut of size at most $s -2$.
\end{lemma}
\begin{proof}
The supercycle with $s$ vertices, $A$, has at least two vertices of degree $3$.
The size of the $(A, G-A)$ cut is thus at most $s -2$. 
This cut need not be an $M$-cut because the edges may have a common neighbor
in $G-A$. To repair this, we will now add, iteratively, 
the common neighbors of edges in the cut to $A$, 
until no edges have a common neighbor in $G-A$.
Note that in any iteration, if a vertex, $v$, adjacent to exactly two cut-edges
was chosen, then the size of $A$ increases by $1$ and the size of the cut decreases
by $1$ (since, these two cut-edges will get added to $A$ along with $v$,  but
 since the graph is cubic, the third edge from $v$ will become a part of the
cut-edges).
If a vertex adjacent to three cut-edges was chosen, then the size of $A$ increases
by $1$ while the number of cut-edges decreases by $3$.
From this we can see that the maximum number of vertices that could
have been added to $A$ during this process is $s -3$. 
Now there are three conditions to check.

The first condition is that this process returns
a non-empty second component. This would occur if 
$$ (n- s) - (s-3)>0$$
or,
$$ n > 2s -3.$$

The second condition is that the second component should not be a collection of disjoint cycles.
For this we note that it is enough to check that at every stage, the 
number of cut-edges is strictly smaller than the number of vertices
in $G-A$. But since in the above iterations, the number of cut-edges decreases
by a number greater than or equal to the decrease in the size of $G-A$, it is enough to 
check that before the iterations, the number of cut-edges is strictly 
smaller than the number of vertices in $G-A$.
This is the condition
$$n - s > s-2$$
or,
$$ n>2s -2.$$

Note that if this inequality holds then the non-emptiness condition will 
also hold.

Finally, we need to check that both components are connected. $A$ is connected
but $G-A$ need not be. We pick a component 
in $G-A$ that has more vertices than the number of cut-edges adjacent to it.
Since the number of cut-edges is strictly smaller than the number of vertices in
$G-A$, there must be one such component, say $B$, in $G-A$. We add every other
component of $G-A$ to $A$. Note that the size of the cut only decreases with
this step. Since $B$ is connected and has more vertices than the number of
cut-edges, $B$ cannot be a cycle.
\end{proof}

\begin{corollary}\label{exisMcut}
Every connected cubic graph on $n \ge 18$ vertices contains a suitable $M$-cut.
\end{corollary}
\begin{proof}
Using the first two lemmas, we have a supercycle with $s\le 2 \lceil \log (\frac{n+1}{g}) \rceil +g-1$ vertices where $3\le g\le 2\lceil \log ( \frac{n}{3} +1) \rceil $. Then using the last lemma, we have an $M$-cut with both partitions being a supercycle if $n>2s-2$. So all we need to check is that $n$ is indeed big enough.
Note that

$$s\le 2 \log (\frac{n+1}{g}) +g+1= 2 \log (n+1) + g+1 -2\log g \le $$
$$ \le 2 \log (n+1) + 2 \log ( \frac{n}{3} +1) -2\log (2 \log ( \frac{n}{3} +1))+1$$

where the last inequality follows from the fact that $x-2\log x$ is increasing for $x\ge 2/\log_e 2\approx 2.88$. So we can bound the right hand side from above by $4 \log (n+1) +1$. Now we need that 

$$n> 2(4 \log (n+1) +1)-2=8\log (n+1)$$

which holds if $n\ge 44$.

The statement can be checked for $18\le n\le 42$ with code that can be found in the Appendix. It outputs for a given value of $n$, the $g$ for which $2s-2$ is maximum and this maximum value.
Based on the output we can see that for $n \ge 18$, this value is smaller.
\end{proof}

\subsection{Proof}

\begin{lemma}\label{cubicdr}
Let $G$ be a connected cubic graph with a suitable $M$-cut. 
Then, $G$ can be drawn with the four basic slopes.
\end{lemma}

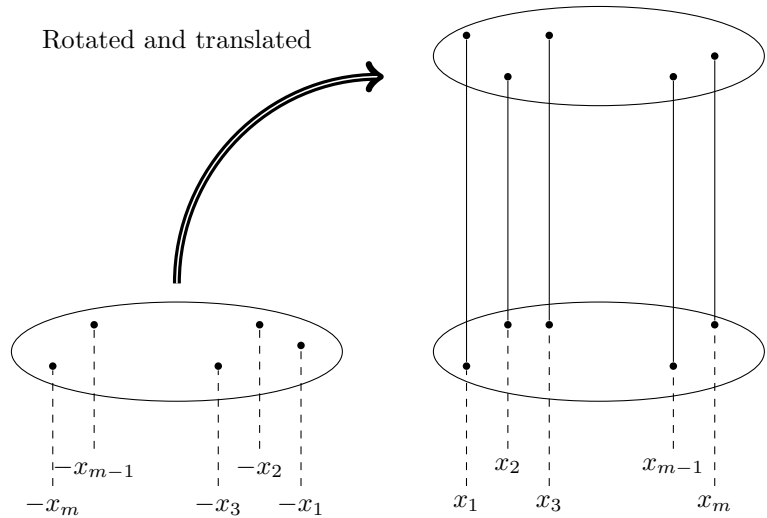
\begin{figure*}[h]
{\centering
\begin{tikzpicture}[scale=0.55]
\node [fill=black,circle,inner sep=1pt] (1) at (1,1) {}; 
\node [fill=black,circle,inner sep=1pt] (2) at (2,2) {}; 
\node [fill=black,circle,inner sep=1pt] (3) at (5,1) {}; 
\node [fill=black,circle,inner sep=1pt] (4) at (6,2) {}; 
\node [fill=black,circle,inner sep=1pt] (5) at (7,1.5) {}; 
\node [fill=black,circle,inner sep=1pt] (6) at (11,1) {}; 
\node [fill=black,circle,inner sep=1pt] (7) at (12,2) {}; 
\node [fill=black,circle,inner sep=1pt] (8) at (13,2) {}; 
\node [fill=black,circle,inner sep=1pt] (9) at (16,1) {}; 
\node [fill=black,circle,inner sep=1pt] (10) at (17,2) {}; 
\node [fill=black,circle,inner sep=1pt] (11) at (11,9) {}; 
\node [fill=black,circle,inner sep=1pt] (12) at (12,8) {}; 
\node [fill=black,circle,inner sep=1pt] (13) at (13,9) {}; 
\node [fill=black,circle,inner sep=1pt] (14) at (16,8) {}; 
\node [fill=black,circle,inner sep=1pt] (15) at (17,8.5) {}; 

\draw [black] (8,1.35) arc (0:360:4cm and 1.2cm);
\draw [black] (18.2,1.35) arc (0:360:4cm and 1.2cm);
\draw [black] (18.2,8.5) arc (0:360:4cm and 1.2cm);

\draw [dashed,black] (11,-1.9) node[below] {$x_1$} -- (6);
\draw [dashed,black] (12,-1) node[below] {$x_2$} -- (7);
\draw [dashed,black] (13,-1.9) node[below] {$x_3$} -- (8);
\draw [dashed,black] (16,-1) node[below] {$x_{m-1}$} -- (9);
\draw [dashed,black] (17,-1.9) node[below] {$x_m$} -- (10);
\draw [dashed,black] (1,-1.9) node[below] {$-x_m$} -- (1);
\draw [dashed,black] (2,-1) node[below] {$-x_{m-1}$} -- (2);
\draw [dashed,black] (5,-1.9) node[below] {$-x_3$} -- (3);
\draw [dashed,black] (6,-1) node[below] {$-x_2$} -- (4);
\draw [dashed,black] (7,-1.9) node[below] {$-x_1$} -- (5);
\draw [black] (6) -- (11);
\draw [black] (7) -- (12);
\draw [black] (8) -- (13);
\draw [black] (9) -- (14);
\draw [black] (10) -- (15);

\draw [very thick,double,->] (4,3) node[above=3cm] {Rotated and translated} arc (180:90:5cm);

\end{tikzpicture}
\caption[Patching components of $M$-cut]{The $x$-coordinates of the degree $2$ vertices is suitably chosen and one component is rotated and translated to make the $M$-cut vertical.}
} 
\label{fig:final}
\end{figure*}

\begin{proof}
The proof follows rather straightforwardly from Theorem \ref{slopenum2}. 
Note that the two components are subcubic graphs and we can choose 
the $x$-coordinates of the vertices of the $M$-cut (since they are 
the vertices with degree two in the components). If we picked coordinates 
$x_1, x_2, \ldots, x_m$ in one component, then for the neighbors of these vertices in the other component we pick the $x$-coordinates $-x_1,-x_{2},\ldots, -x_m$. We now rotate the second
component by $\pi$ and place it very high above the other component so that the drawings 
of the components do not intersect
and align them so that the edges of the $M$-cut will be vertical (slope $\pi/2$). 
Also, since Theorem \ref{slopenum2} guarantees that degree two vertices have 
no other vertices on the vertical line above them, hence the drawing
we obtain above is a valid representation of $G$ with the basic slopes.
\end{proof}

By combining Lemma \ref{exisMcut} and Lemma \ref{cubicdr}, we can see that Theorem \ref{thm_4basic_slopes} 
is true for all cubic graphs with $n\ge 18$. For smaller graphs, we 
reduce the number of graphs we have to check with the 
help of Lemma \ref{vert_conn} and Remark \ref{conn_remark} as a 
consequence of which, a graph that cannot be drawn with the four basic 
slopes must be three vertex and edge connected. 




We also employ the following theorem by Max Engelstein \cite{eng}.

\begin{lemma}\label{maxeng}
Every $3$-connected cubic graph with a Hamiltonian cycle can be drawn 
in the plane with the four basic slopes.
\end{lemma}

Note that combining this with Lemma \ref{vert_conn} 
we even get

\begin{corollary} 
Every cubic graph with a Hamiltonian cycle can be drawn 
in the plane with the four basic slopes.
\end{corollary}

The graphs which now need to be checked satisfy the following conditions:
\begin{enumerate}
\item the number of vertices is at most $16$
\item the graph is $3$-connected
\item the graph does not have a Hamiltonian cycle.
\end{enumerate}

\begin{remark}
We now bring the attention of our reader to Remark \ref{triangle_free_remark} to add that we may also add to the above list that the graph does not contain a triangle. However, we use our girth lemmas to have an easy way to analyze the 
graphs excluded by only the above three assertions.
\end{remark}

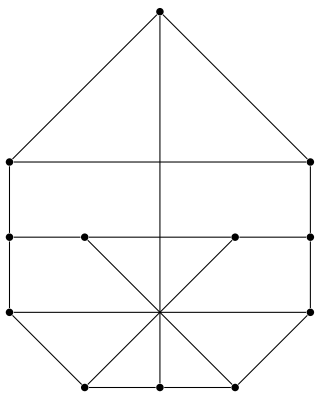
\begin{figure}[htp]\label{fig:tietze}
{\centering
\begin{tikzpicture}[scale=1]
\node [fill=black,circle,inner sep=1pt] (1) at (0,1) {}; 
\node [fill=black,circle,inner sep=1pt] (2) at (0,2) {}; 
\node [fill=black,circle,inner sep=1pt] (3) at (1,2) {}; 
\node [fill=black,circle,inner sep=1pt] (4) at (3,2) {}; 
\node [fill=black,circle,inner sep=1pt] (5) at (4,2) {}; 
\node [fill=black,circle,inner sep=1pt] (6) at (4,1) {}; 
\node [fill=black,circle,inner sep=1pt] (7) at (3,0) {}; 
\node [fill=black,circle,inner sep=1pt] (8) at (2,0) {}; 
\node [fill=black,circle,inner sep=1pt] (9) at (1,0) {}; 
\node [fill=black,circle,inner sep=1pt] (10) at (0,3) {}; 
\node [fill=black,circle,inner sep=1pt] (11) at (4,3) {}; 
\node [fill=black,circle,inner sep=1pt] (12) at (2,5) {}; 
\draw [black] (1) -- (2) -- (3) -- (4) -- (5) -- (6) -- (1);
\draw [black] (1) -- (9) -- (8) -- (7) -- (6);
\draw [black] (10) -- (11) -- (12) -- (10);
\draw [black] (9) -- (4);
\draw [black] (7) -- (3);
\draw [black] (2) -- (10);
\draw [black] (8) -- (12);
\draw [black] (5) -- (11);
\end{tikzpicture}
\caption[The Tietze's graph]{The Tietze's graph drawn with the four basic slopes.}
} 
\end{figure}

Note that if the number of vertices is at most $16$, then it follows from Lemma \ref{girth} that the girth is at most $6$.
Luckily there are several lists available of cubic graphs with a given number of vertices, $n$ and a given girth, $g$.

If $g=6$, then there are only two graphs with at most $16$ vertices (see \cite{graphlist, meringer}), both containing a Hamiltonian cycle.

If $g=5$ and $n=16$, then Lemma \ref{supercycle} gives a supercycle with at most $8$ vertices, so using Lemma \ref{Mcut} we are done.

If $g=5$ and $n=14$, then there are only nine graphs (see \cite{graphlist, meringer}), all containing a Hamiltonian cycle.

If $g\le 4$ and $n=16$, then Lemma \ref{supercycle} gives a supercycle with at most $8$ vertices, so using Lemma \ref{Mcut} we are done.

If $g\le 4$ and $n=14$, then Lemma \ref{supercycle} gives a supercycle with at most $7$ vertices, so using Lemma \ref{Mcut} we are done.

Finally, all graphs with at most $12$ vertices are either not $3$-connected or contain a Hamiltonian cycle, except for the Petersen graph and Tietze's Graph (see \cite{wiki}).
For the drawing of these two graphs, see the respective Figures.

%% file: 1_5_final_remarks.tex
\section{Which four slopes? and other concluding questions}
After establishing Theorem \ref{thm_4basic_slopes} the question arises whether we could have used any other four slopes.
Call a set of slopes {\em good} if every cubic graph has a straight-line drawing with them. 
In this section we prove Theorem \ref{karakterizacio} that claims that the following statements are equivalent for a set $S$ of four slopes.

\begin{enumerate}
\item $S$ is good.
\item $S$ is an affine image of the four basic slopes.
\item We can draw $K_4$ with $S$.
\end{enumerate}

\begin{proof}
Since affine transformation keeps incidences, any set that is the affine image of the four basic slopes is good.

On the other hand, if a set $S=\{s_1,s_2,s_3,s_4\}$ is good, then $K_4$ has a straight-line drawing with $S$. Since we do not allow a vertex to be in the interior of an edge, the four vertices must be in general position. This implies that two incident edges cannot have the same slope. Therefore there are two slopes, without loss of generality $s_1$ and $s_2$, such that we have two edges of each slope. These four edges must form a cycle of length four, which means that the vertices are the vertices of a parallelogram. But in this case there is an affine transformation that takes the parallelogram to a square. This transformation also takes $S$ into the four basic slopes.
\end{proof}

Note that a similar reasoning shows that no matter how many slopes we take, their set need not be good, because we cannot even draw $K_4$ with them unless they satisfy some correlation. 
The above proofs use the four basic slopes only in a few places (for rotation invariance and to start induction). Thus, we make the following conjecture.

\begin{conjecture}
There is a (not necessarily connected, finite) graph such that a set of slopes is good if and only if this graph has a straight-line drawing with them.
\end{conjecture}

This finite graph would be the disjoint union of $K_4$, maybe the Petersen graph and other small graphs. We could not even rule out the possibility that $K_4$ (or maybe another, connected graph) is alone sufficient. Note that we can define a partial order on the graphs this way. Let $G < H$ if any set of slopes that can be used to draw $H$ can also be used to draw $G$. This way of course $G\subset H \Rightarrow G<H$ but what else can we say about this poset?

Is it possible to use this new method to prove that the slope parameter of cubic graphs is also four?

The main question remains to prove or disprove whether the slope number of graphs with maximum degree four is unbounded.

%% file: 2_1_obstacle_introduction.tex
\section{Introduction}

Consider a set $P$ of points in the plane and a set of closed polygonal obstacles whose vertices together with the points in $P$ are in {\em general position}, that is, no {\em three} of them are on a line. The corresponding {\em visibility graph} has $P$ as its vertex set, two points $p,q\in P$ being connected by an edge if and only if the segment $pq$ does not meet any of the obstacles. Visibility graphs are extensively studied and used in computational geometry, robot motion planning, computer vision, sensor networks, etc.; see \cite{BKOS00}, \cite{G07}, \cite{OR97}, \cite{O99}, \cite{Ur00}.

Recently, Alpert, Koch, and Laison \cite{AKL09} introduced an interesting new parameter of graphs, closely related to visibility graphs. Given a graph $G$, we say that a set of points and a set of polygonal obstacles as above constitute an {\em obstacle representation} of $G$, if the corresponding visibility graph is isomorphic to $G$. A representation with $h$ obstacles is also called an $h$-obstacle representation. The smallest number of obstacles in an obstacle representation of $G$ is called the {\em obstacle number} of $G$ and is denoted by ${\rm obs}(G)$. If we are allowed to use only {\em convex} obstacles, then the corresponding parameter ${\rm obs}_{c}(G)$ is called the {\em convex obstacle number} of $G$. Of course, we have ${\rm obs}(G)\leq {\rm obs}_c(G)$ for every $G$, but the two parameters can be very far apart.

A special instance of the obstacle problem has received a lot of attention, due to its connection to the Szemer\'edi-Trotter theorem on incidences between points and lines~\cite{ST83a}, \cite{ST83b}, and other classical problems in incidence geometry~\cite{PA95}. It is an exciting open problem to decide whether the obstacle number of ${\overline{K}}_n$, the empty graph on $n$ vertices, is $O(n)$ if the obstacles must be {\em points}. The best known upper bound is
$n2^{O(\sqrt{\log n})}$; see
Pach
\cite{Pach03}, Dumitrescu et al.~\cite{DPT09}, Matou\v sek~\cite{M09}, and Aloupis et al. \cite{A+10conf}.

Alpert et al.~\cite{AKL09} constructed a bipartite graph and a split graph (a graph whose vertex set is the union of a complete graph and an independent set), both with a fixed number of vertices, with obstacle number
at least
{\em two}. In \cite{mps10} another graph, whose vertex set is the union of two complete subgraphs, was shown to have obstacle number at least two. 
Consequently, no graph of obstacle number {\em one} can contain a subgraph isomorphic to these graphs. Using this and some extremal graph theoretic tools developed by Erd\H os, Kleitman, Rothschild, Frankl, R\"odl, Pr\"omel, Steger, Bollob\'as, Thomason and others, the following was proved.

\begin{theorem}[\cite{mps10}]
\label{ult}
For any fixed positive integer $h$, the number of graphs on $n$ (labeled) vertices with obstacle number at most $h$ is at most $2^{o(n^2)}.$
\end{theorem}

Since the number of bipartite graphs with $n$ labeled vertices is $\Omega(2^{n^2/4})$, this also implies that there exist bipartite graphs with arbitrarily large obstacle number.

For every sufficiently large $n$, Alpert et al. constructed a 
graph with $n$ vertices with obstacle number at least $\Omega\left(\sqrt{\log n}\right)$. 
By using the existence of graphs with obstacle number at least $2$  and
a result by Erd{\H{o}}s and Hajnal \cite{EH89},
we show the existence of graphs with much larger obstacle numbers.

\begin{theorem}
\label{thm_obstacle1}
For every $\varepsilon>0$, there exists an integer $n_0=n_0(\varepsilon)$ such that for all $n\geq n_0$, there are graphs $G$ on $n$ vertices such that their obstacle numbers satisfy
$${\rm obs}(G)\geq \Omega\left(n^{1-\varepsilon}\right).$$
\end{theorem}


In Section~\ref{padmini}, we improve on the last two corollaries, using some estimates on the number of different {\em order types} of $n$ points in the Euclidean plane, discovered by Goodman and Pollack~\cite{GP86}, \cite{GP93} (see also Alon~\cite{Al86}). We establish the following results.

\begin{theorem}
\label{enumeration}
For any fixed positive integer $h$, the number of graphs on $n$ (labeled) vertices with obstacle number at most $h$ is at most $$2^{O(hn\log^2n)}.$$
\end{theorem}

\begin{theorem}\label{thm_obstacle2}
For every $n$, there exist graphs $G$ on $n$ vertices with obstacle numbers
$${\rm obs}(G)\geq \Omega\left({n}/{\log^2n}\right).$$
\end{theorem}

Note that the last statement directly follows from Theorem~\ref{enumeration}. Indeed, since the total number of (labeled) graphs with $n$ vertices is $2^{\Omega(n^2)}$, as long as
$2^{O(hn\log^2n)}$ is smaller that this quantity, there is a graph with obstacle number larger than $h$.

We prove a slightly better bound for convex obstacle numbers.

\begin{theorem}\label{convex}
For every $n$, there exist graphs $G$ on $n$ vertices with convex obstacle numbers
$${\rm obs}_c(G)\geq \Omega\left({n}/{\log n}\right).$$
\end{theorem}

If we only allow segment obstacles, we get an even better bound.
Following Alpert et al., we define the {\em segment obstacle number} ${\rm obs}_s(G)$ of a graph $G$ as the minimal number of obstacles in an obstacle representation of $G$, in which each obstacle is a straight-line segment.

\begin{theorem}\label{segment}
For every $n$, there exist graphs $G$ on $n$ vertices with segment obstacle numbers
$${\rm obs}_s(G)\geq \Omega\left({n^2}/{\log n}\right).$$
\end{theorem}


We then improve the bound for the general obstacle number as follows.


\begin{theorem}\label{concave}
For every $n$, there exists a graph $G$ on $n$ vertices with
obstacle number
$${\rm obs}(G)\geq \Omega\left({n}/{\log n}\right).$$
\end{theorem}

This comes close to answering the question in \cite{AKL09} whether
there exist graphs with $n$ vertices and obstacle number at least
$n$. However, we have no upper bound on the maximum obstacle
number of $n$-vertex graphs, better than $O(n^2)$.




Given any placement (embedding) of the vertices of $G$ in general position in the plane, a {\em drawing} of $G$ consists of the image of the embedding and the set of {\em open segments} connecting all pairs of points that correspond to the edges of $G$.  If there is no danger of confusion, we make no notational difference between the vertices of $G$ and the corresponding points, and between the pairs $uv$ and the corresponding open segments. The complement of the set of all points that correspond to a vertex or belong to at least one edge of $G$ falls into connected components. These components are called the {\em faces} of the drawing. Notice that if $G$ has an obstacle representation with a particular placement of its vertex set, then

(1) each obstacle must lie entirely in one face of the drawing, and

(2) each non-edge of $G$ must be blocked by at least one of the obstacles.

%% file: 2_2_obstacleconf.tex
\section{Extremal methods and proof of Theorem \ref{thm_obstacle1}}

In order to prove Theorem~\ref{thm_obstacle1}, we need the following result, which shows that if $G$ avoids at least one induced subgraph with $k$ vertices, for some $k\ll \log n$, then the Erd\H os-Szekeres bound on ${\rm hom}(G)$ can be substantially improved. We note that $hom(G)$ for a graph $G$ is defined as the size of the largest clique or independent set in the graph. Also, a graph is $k$-universal if it contains every graph on $k$ vertices as induced subgraph.

\begin{theorem}[Erd\H os, Hajnal \cite{EH89}]\label{ehthm}
For any fixed positive integer $t$, there is an $n_0 = n_0(t)$ with the following property. Given any graph $G$ on $n>n_0$ vertices and any integer $k<2^{c\sqrt{log n}/t}$, either $G$ is $t$-universal or we have ${\rm hom}(G) \ge k$. (Here $c>0$ is a suitable constant.)
\end{theorem}

We now prove Theorem~\ref{thm_obstacle1}.

\medskip
\begin{proof}
For the sake of clarity of the presentation, we systematically omit all floor and ceiling functions wherever they are not essential. Let $H$ be a graph of $t$ vertices that does not admit a $1$-obstacle representation. Fix any $0<\varepsilon<1$, and choose an integer $N\ge n_0$, that satisfies the inequality
\begin{equation}\label{szamozott}
2^{c\sqrt{\varepsilon\log N}/t}  > 2\log N,
\end{equation}
where $c,n_0$ are constants that appear in the previous theorem.

For any $n\ge N$, we set $m=n^{1-\varepsilon}$.
According to a theorem of Erd\H os~\cite{Er47}, there exists a graph $G$ with $n$
vertices such that
$${\rm hom}(G)<2\log n< 2^{c\sqrt{\log (n/m)}/t}.$$

\input{FigSources/nepsObstacles.tex}

Consider an obstacle representation of $G$ with the smallest number $h$ of obstacles. Suppose without loss of generality that in our coordinate system all points of $G$ have different $x$-coordinates.  By vertical lines, partition the plane into $m$ strips, each  containing $n/m$ points. Let $G_i$ denote the subgraph of $G$ induced by the vertices lying in the $i$-th strip $(1\le i\le m)$.

Obviously, we have
$${\rm hom}(G_i)\le {\rm hom}(G)<2^{c \sqrt{\log (n/m)}/t},$$
for every $i$. Hence, applying Theorem~\ref{ehthm} to each $G_i$ separately, we conclude that each must be $t$-universal. In particular, each $G_i$ contains an induced subgraph isomorphic to $H$. That is, we have ${\rm obs}(G_i)>1$ for every $i$, which means that each $G_i$ requires at least {\em two} obstacles.

As was explained at the end of the Introduction, each obstacle must be contained in an interior or in the exterior face of the graph. Therefore, in an $h$-obstacle representation of $G$, each $G_i$ must have at least one internal face that contains an obstacle, and there must be at least one additional obstacle (which may possibly contained in the interior face of every $G_i$). At any rate, we have $h>m=n^{1-\varepsilon},$ as required.  

\end{proof}

\section{Encoding graphs of low obstacle number}
\label{padmini}

The aim of this section is to prove Theorems~\ref{enumeration}--\ref{segment}. The idea is to find a short encoding of the obstacle representations of graphs, and to use this to give an upper bound on the number of graphs with low obstacle number. 

We need to review some simple facts from combinatorial geometry. Two sets of points, $P_1$ and $P_2$, in general position in the plane are said to have the same {\em order type} if there is a one to one correspondence between them with the property that the orientation of any triple in $P_1$ is the same as the orientation of the corresponding triple in $P_2$. Counting the number of different order types is a classical task, see e.g. 

\begin{theorem}[Goodman, Pollack~\cite{GP86}]\label{GoodmanPollack}
The number of different order types of $n$ points in general position in the plane is $2^{O(n\log n)}$.
\end{theorem}

\noindent Observe that the same upper bound holds for the number of different order types of $n$ {\em labeled} points, because the number of different permutations of $n$ points is $n!=2^{O(n\log n)}$.

In a graph drawing, the \emph{complexity} of a face is the number of line segment sides bordering it.
The following result was proved by Arkin, Halperin, Kedem, Mitchell, and Naor (see Matou\v sek, Valtr~\cite{MV97} for its sharpness).

\begin{theorem}[Arkin et al.~\cite{AHK95}]\label{Arkin}
The complexity of a single face in a drawing of a graph with $n$ vertices is at most $O(n\log n)$.
\end{theorem}

\noindent Note that this bound does not depend of the number of edges of the graph. We are now ready to prove Theorem~\ref{enumeration}.

\medskip

\begin{proof}
For any graph $G$ with $n$ vertices that admits an $h$-obstacle representation, fix such a representation. Consider the visibility graph $G$ of the vertices in this representation. As explained at the end of the Introduction, any obstacle belongs to a single face in this drawing. In view of Theorem~\ref{Arkin}, the complexity of every face is $O(n\log n)$. Replacing each obstacle by a slightly shrunken copy of the face containing it, we can achieve that every obstacle {\em is} a polygonal region with $O(n\log n)$ sides. 

Notice that the order type of the sequence $S$ starting with the vertices of $G$, followed by the vertices of the obstacles (listed one by one, in cyclic order, and properly separated from one another) completely determines $G$. That is, we have a sequence of length $N$ with $N\le n + c_1 h n\log n$. According to 
Theorem~\ref{GoodmanPollack} (and the following comment), the number of different order types with this many points is at most
$$2^{O(N\log N)}<2^{chn\log^2 n},$$
for a suitable constant $c>0$. This is a very generous upper bound: most of the above sequences do not correspond to any visibility  graph $G$.
\end{proof}

If in the above proof the average number of sides an obstacle can have is small, then we obtain

\begin{theorem}\label{sidelemma}
The number of graphs admitting an obstacle representation with at most $h$ obstacles, having a total of at most $hs$ sides, is at most
$$ 2^{O(n\log n +  hs\log (hs))}.$$
\end{theorem}
 
In particular, for segment obstacles ($s=2$), Theorem~\ref{sidelemma} immediately implies Theorem~\ref{segment}. Indeed, as long as the bound in Theorem~\ref{sidelemma} is smaller than $2^{n\choose 2}$, the total number of graphs on $n$ labeled vertices, we can argue that there is a graph with segment obstacle number larger than $h$.

We now show how to prove Theorem \ref{convex} with an easier way to encode the drawing of a graph and its convex obstacles.


\begin{proof}
As before, it is enough to bound the number of graphs that admit an obstacle representation with at most $h$ convex obstacles. Let us fix such a graph $G$, together with a representation. Let $V$ be the set of points representing the vertices, and let $O_1,\ldots, O_h$ be the convex obstacles. For any obstacle $O_i$, rotate an oriented tangent line $\ell$ along its boundary in the clockwise direction. We can assume without loss of generality that $\ell$ never passes through two points of $V$. Let us record the sequence of points met by $\ell$. If $v\in V$ is met at the right side of $\ell$, we add the symbol $v_+$ to the sequence, otherwise we add $v_-$ (Figure 2.2).
When $\ell$ returns to its initial position, we stop. The resulting sequence consists of $2n$ characters. From this sequence, it is easy to reconstruct which pairs of vertices are visible in the presence of the single obstacle $O_i$. Hence, knowing these sequences for every obstacle $O_i$, completely determines the visibility graph $G$. The number of distinct sequences assigned to a single obstacle is at most $(2n)!$, so that the number of graphs with convex obstacle number at most $h$ cannot exceed $((2n)!)^h/h!<(2n)^{2hn}$. As long as this number is smaller than $2^{n\choose 2}$, there is a graph with convex obstacle number larger than $h$. 
\end{proof} 
\input{FigSources/convexObstacle.tex}

%% file: FigSources/nepsObstacles.tex
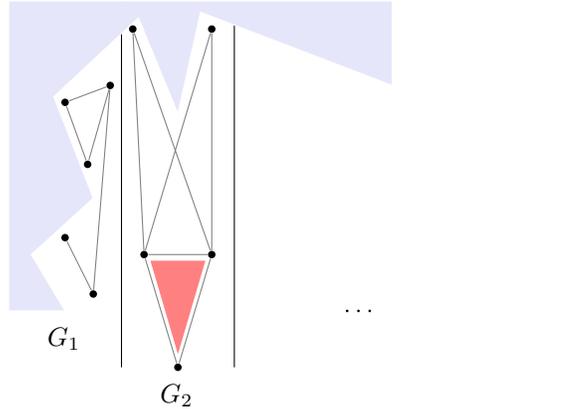
\begin{figure*}[htp]
{\centering
\begin{tikzpicture}[scale=0.75]
\draw (2,2) -- (2,8.5);
\draw (4,2) -- (4,8.5);
\node [fill=black,circle,inner sep=0.1pt] at (6,3) {}; 
\node [fill=black,circle,inner sep=0.1pt] at (6.2,3) {}; 
\node [fill=black,circle,inner sep=0.1pt] at (6.4,3) {}; 
\node [label=left:$G_1$] at (1.6,2.5) {}; 
\node [label=left:$G_2$] at (3.6,1.5) {}; 
\draw (10,2) -- (10,8.5);
\filldraw [draw=white, fill=blue!80!black!10!white] (0,3) -- (0,8.5) -- (6.8,8.5) -- (6.8,7) -- (3.4,8.3) -- (3,6.5) -- (2.3,8.2) -- (0.8,6.8) -- (1.5,5) -- (0.4,4) -- (1,3) -- (0,3) -- cycle;
\node [fill=black,circle,inner sep=1pt] (1) at (3.6,8) {}; 
\node [fill=black,circle,inner sep=1pt] (2) at (2.2,8) {}; 
\node [fill=black,circle,inner sep=1pt] (3) at (2.4,4) {}; 
\node [fill=black,circle,inner sep=1pt] (4) at (3.6,4) {}; 
\node [fill=black,circle,inner sep=1pt] (5) at (3,2) {}; 
\node [fill=black,circle,inner sep=1pt] (6) at (1.8,7) {}; 
\node [fill=black,circle,inner sep=1pt] (7) at (1.0,6.7) {}; 
\node [fill=black,circle,inner sep=1pt] (8) at (1.4,5.6) {}; 
\node [fill=black,circle,inner sep=1pt] (9) at (1.0,4.3) {}; 
\node [fill=black,circle,inner sep=1pt] (10) at (1.5,3.3) {}; 

\draw [gray] (1) -- (3) -- (4) -- (2) -- (3) -- (5) -- (4) -- (1);
\draw [gray] (6) -- (7) -- (8) -- (6) -- (10) -- (9);
\filldraw [draw=white, fill=red!50!white] (2.5,3.9) -- (3.5,3.9) -- (3,2.2) -- (2.5,3.9) -- cycle;
\end{tikzpicture}
\caption[Division of graph into parts with interior obstacle]{Division of the graph into $m$ parts each with $n/m$ points. The light grey obstacle is a common exterior obstacle, while the darker one is an internal obstacle of $G_2$.} 
}
\label{fig:nepsOb}
\end{figure*}

%% file: FigSources/convexObstacle.tex
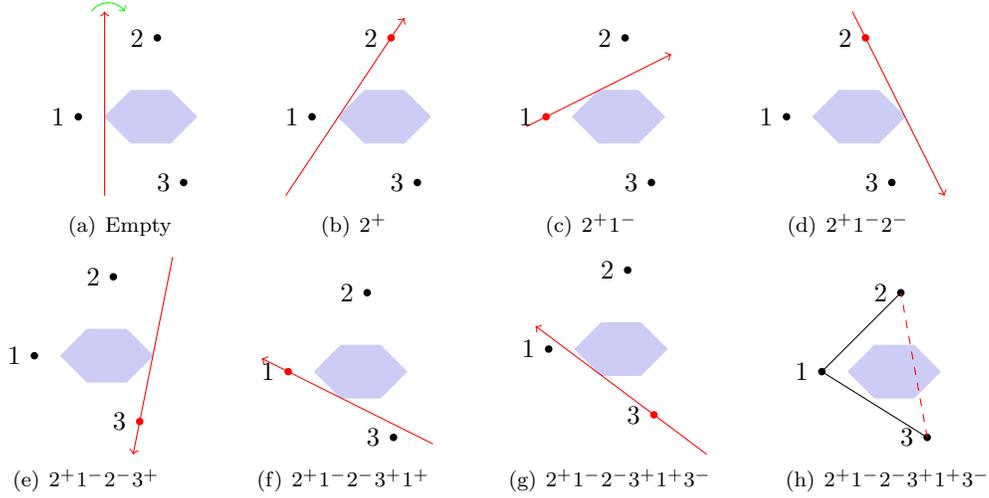
\begin{figure*}[htp]
{\centering
\subfigure[Empty]{
\begin{tikzpicture}[scale=0.35]
\filldraw [blue!80!black!20!white] (0,0) -- (1,1) -- (2.5,1) -- (3.5,0) -- (2.5,-1) -- (1,-1) -- (0,0) -- cycle;
\node [fill=black,circle,inner sep=1pt,label=left:$1$] (1) at (-1,0) {}; 
\node [fill=black,circle,inner sep=1pt,label=left:$2$] (2) at (2,3) {}; 
\node [fill=black,circle,inner sep=1pt,label=left:$3$] (3) at (3,-2.5) {}; 
\draw [<-,red] (0,4) -- (0,-3);
\draw [->,green] (-0.5,4) arc (145:35:22pt);
\end{tikzpicture}
}
\qquad
\subfigure[$2^+$]{
\begin{tikzpicture}[scale=0.35]
\filldraw [blue!80!black!20!white] (0,0) -- (1,1) -- (2.5,1) -- (3.5,0) -- (2.5,-1) -- (1,-1) -- (0,0) -- cycle;
\node [fill=black,circle,inner sep=1pt,label=left:$1$] (1) at (-1,0) {}; 
\node [fill=red,circle,inner sep=1pt,label=left:$2$] (2) at (2,3) {}; 
\node [fill=black,circle,inner sep=1pt,label=left:$3$] (3) at (3,-2.5) {}; 
\draw [->,red] (-2,-3) -- (2.5,3.75);
\end{tikzpicture}
}
\qquad
\subfigure[$2^+1^-$]{
\begin{tikzpicture}[scale=0.35]
\filldraw [blue!80!black!20!white] (0,0) -- (1,1) -- (2.5,1) -- (3.5,0) -- (2.5,-1) -- (1,-1) -- (0,0) -- cycle;
\draw [->,red] (-1.75,-0.375) -- (3.75,2.375);
\node [fill=red,circle,inner sep=1pt,label=left:$1$] (1) at (-1,0) {}; 
\node [fill=black,circle,inner sep=1pt,label=left:$2$] (2) at (2,3) {}; 
\node [fill=black,circle,inner sep=1pt,label=left:$3$] (3) at (3,-2.5) {}; 
\end{tikzpicture}
}
\qquad
\subfigure[$2^+1^-2^-$]{
\begin{tikzpicture}[scale=0.35]
\filldraw [blue!80!black!20!white] (0,0) -- (1,1) -- (2.5,1) -- (3.5,0) -- (2.5,-1) -- (1,-1) -- (0,0) -- cycle;
\node [fill=black,circle,inner sep=1pt,label=left:$1$] (1) at (-1,0) {}; 
\node [fill=red,circle,inner sep=1pt,label=left:$2$] (2) at (2,3) {}; 
\node [fill=black,circle,inner sep=1pt,label=left:$3$] (3) at (3,-2.5) {}; 
\draw [->,red] (1.5,4) -- (5,-3);
\end{tikzpicture}
}
\qquad
\subfigure[$2^+1^-2^-3^+$]{
\begin{tikzpicture}[scale=0.35]
\filldraw [blue!80!black!20!white] (0,0) -- (1,1) -- (2.5,1) -- (3.5,0) -- (2.5,-1) -- (1,-1) -- (0,0) -- cycle;
\node [fill=black,circle,inner sep=1pt,label=left:$1$] (1) at (-1,0) {}; 
\node [fill=black,circle,inner sep=1pt,label=left:$2$] (2) at (2,3) {}; 
\node [fill=red,circle,inner sep=1pt,label=left:$3$] (3) at (3,-2.5) {}; 
\draw [<-,red] (2.75,-3.75) -- (4.25,3.75);
\end{tikzpicture}
}
\qquad
\subfigure[$2^+1^-2^-3^+1^+$]{
\begin{tikzpicture}[scale=0.35]
\filldraw [blue!80!black!20!white] (0,0) -- (1,1) -- (2.5,1) -- (3.5,0) -- (2.5,-1) -- (1,-1) -- (0,0) -- cycle;
\node [fill=red,circle,inner sep=1pt,label=left:$1$] (1) at (-1,0) {}; 
\node [fill=black,circle,inner sep=1pt,label=left:$2$] (2) at (2,3) {}; 
\node [fill=black,circle,inner sep=1pt,label=left:$3$] (3) at (3,-2.5) {}; 
\draw [<-,red] (-2,0.5) -- (4.5,-2.75);
\end{tikzpicture}
}
\qquad
\subfigure[$2^+1^-2^-3^+1^+3^-$]{
\begin{tikzpicture}[scale=0.35]
\filldraw [blue!80!black!20!white] (0,0) -- (1,1) -- (2.5,1) -- (3.5,0) -- (2.5,-1) -- (1,-1) -- (0,0) -- cycle;
\node [fill=black,circle,inner sep=1pt,label=left:$1$] (1) at (-1,0) {}; 
\node [fill=black,circle,inner sep=1pt,label=left:$2$] (2) at (2,3) {}; 
\node [fill=red,circle,inner sep=1pt,label=left:$3$] (3) at (3,-2.5) {}; 
\draw [<-,red] (-1.5,0.875) -- (5,-4);
\end{tikzpicture}
}
\qquad
\subfigure[$2^+1^-2^-3^+1^+3^-$]{
\begin{tikzpicture}[scale=0.35]
\filldraw [blue!80!black!20!white] (0,0) -- (1,1) -- (2.5,1) -- (3.5,0) -- (2.5,-1) -- (1,-1) -- (0,0) -- cycle;
\node [fill=black,circle,inner sep=1pt,label=left:$1$] (1) at (-1,0) {}; 
\node [fill=black,circle,inner sep=1pt,label=left:$2$] (2) at (2,3) {}; 
\node [fill=black,circle,inner sep=1pt,label=left:$3$] (3) at (3,-2.5) {}; 
\draw (3,-2.5) -- (-1,0) -- (2,3);
\draw [dashed,red](3,-2.5) -- (2,3);
\node (4) at (5,0) {};
\end{tikzpicture}
}
\caption[Constructing a sequence from convex obstacles]{Parts (a) to (g) show the construction of the sequence and (h) shows the visibilities. The arrow on the tangent line indicates the direction from the point of tangency in which we assign $+$ as a label to the vertex. The additional arrow in (a) indicates that the tangent line is rotated clockwise around the obstacle.}
} 
\label{fig:convexOb}
\end{figure*}

%% file: 2_3_obstaclejourn.tex
\section{Proof of Theorem \ref{concave}}\label{sec:concave}
Here we prove Theorem \ref{concave} that claims that for every
$n$, there exists a graph $G$ on $n$ vertices with obstacle number
${\rm obs}(G)\geq \Omega\left({n}/{\log n}\right).$

\begin{proof}
The proof will be a counting argument. From Theorem
\ref{enumeration} we know that the number of graphs on $k$
(labeled) vertices with obstacle number at most one is at most
$2^{o(k^2)}$. Now we will count the graphs with obstacle number
less than $n/2k$. Suppose $G$ has a representation with less than
$n/2k$ obstacles. Fix one such representation. There are $n!$
possibilities for the order of the vertices of $G$ from left to
right (we can suppose that no two are below each other). We divide
the vertices into $n/k$ groups of size $k$, from left to right.
Denote the respective induced graphs by $G_i$.

\begin{claim} At least half of the $G_i$'s require at most one obstacle.
\end{claim}
\begin{proof} By contradiction, suppose that at least half of the $G_i$'s require at least two obstacles. One of each of these obstacles must be in an interior face of the respective $G_i$'s. Thus these obstacles are pair-wise separated by edges and must be different. This together would be at least $n/2k$ obstacles which contradicts the choice of $G$.
\end{proof}

For the subset of $G_i$'s that require at most one obstacle there
are less than $2^{n/k}$ possibilities. Since the number of graphs
on $k$ vertices whose obstacle number is at most one is $2^{o(k^2)}$, the
probability that a $G_i$ has a representation with at most one obstacle is
$2^{o(k^2)-{k\choose 2}}$. Therefore, the probability that a
random graph $G$ has obstacle number at most $n/2k$ is at most
{
$$n!\cdot 2^{n/k}\cdot (2^{o(k^2)-{k\choose 2}})^{n/2k}= 2^{n\log n - \frac{kn}{4}+o(kn)}.$$}

If $k=\Omega\left(5\log n\right)$, this number tends to zero.
Therefore some graphs need at least $\Omega\left({n}/{\log
n}\right)$ obstacles.
\end{proof}

%% file: 2_4_final_remarks.tex
\section{Further properties}

In this section, we describe further properties of obstacle numbers. We start
with another question from \cite{AKL09}.

\begin{theorem}\label{exactly}
For every $h$, there exists a graph with obstacle number exactly
$h$.
\end{theorem}


\begin{proof}
Pick a graph $G$ with obstacle number $h'>h$. (The existence of such a graph follows, e.g., from Corollary~\ref{ult}.) Let $n$ denote the number of vertices of $G$. Consider a complete graph $K_n$ on $V(G)$. Its obstacle number is {\em zero}, and $G$ can be obtained from $K_n$ by successively deleting edges. Observe that as we delete an edge from a graph $G'$, its obstacle number cannot increase by more than {\em one}. This follows from the fact that by blocking the deleted edge with an additional small obstacle that does not intersect any other edge of $G'$, we obtain a valid obstacle representation of the new graph. (Of course, the obstacle number of a graph can also {\em decrease} by the removal of an edge.) Since at the beginning of the process, $K_n$ has obstacle number {\em zero}, at the end $G$ has obstacle number $h'>h$, and whenever it increases, the increase is {\em one}, we can conclude that at some stage we obtain a graph with obstacle number precisely $h$.
\end{proof}

The same argument applies to the convex obstacle number, to the segment obstacle number, and many similar parameters.


Let $H$ be a fixed graph. According to a classical conjecture of Erd\H os and Hajnal~\cite{EH89}, any graph with $n$ vertices that does not have an induced subgraph isomorphic to $H$ contains an independent set or a complete subgraph of size at least $n^{\varepsilon(H)}$, for some positive constant $\varepsilon(H)$. It follows that for any hereditary graph property there exists a constant $\varepsilon>0$ such that every graph $G$ on $n$ vertices with this property satisfies ${\rm hom}(G)\ge n^{\varepsilon}$.

Here we show that the last statement holds for the property that the  graph has bounded obstacle number.

\begin{theorem} For any fixed integer $h>0$, every graph on $n$ vertices with ${\rm obs}_{c}(G) \le h$ satisfies ${\rm hom}(G)\ge \frac{1}{2}n^{\frac{1}{h+1}}$.
\end{theorem}

\begin{proof}
We proceed by induction on $h$. For $h=1$, Alpert et al.~\cite{AKL09} showed that all graphs with convex obstacle number {\em one} are so-called "circular interval graphs" (intersection graphs of a collection of arcs along the circle). It is known that all such graphs $G$ whose maximum complete subgraph is of size $x$ has an independent set of size at least $\frac{n}{2x}$; see \cite{T75}. Setting $x=\sqrt{n/2}$, it follows that ${\rm hom}(G)\geq \frac{1}{2}\sqrt{n}$.

Let $h>1$, and assume that the statement has already been verified for all graphs with convex obstacle number smaller than $h$. Let $G$ be a graph that requires $h$ convex obstacles, and consider one of its representations. Then we have $G = \cap_i G_i$, where $G_i$ denotes the visibility graph of the same set of points after the removal of all but the $i$-th obstacle. 

If the size of the largest independent set in $G_1$ is at least $\frac{1}{2}n^{\frac{1}{h+1}}$, then the statement holds, because this set is also an independent set in $G$. If this is not the case, then, by the above property of circular arc graphs, $G$ must have a complete subgraph $K$ of size at least $n^{\frac{h}{h+1}}$. Consider now the subgraph of $\cap_{i=2}^h G_i$ induced by the vertices of $K$. This graph requires only $h-1$ obstacles. Thus, we can apply the induction hypothesis to obtain that it has a complete subgraph or an independent set of size at least $\frac{1}{2}(n^{\frac{h}{h+1}})^{\frac{1}{h}} = \frac{1}{2}n^{\frac{1}{h+1}}$.
\end{proof}

It is easy to see that every graph $G$ on $n$ vertices with convex obstacle number at most $h$ has the following stronger property, which implies that they satisfy the Erd\H os-Hajnal conjecture: There exists a constant $\varepsilon=\varepsilon(h)$ such that $G$ contains a complete subgraph of size at least $\varepsilon n$ or two sets of size at least $\varepsilon n$ such that no edges between them belongs to $G$ (cf.~\cite{FP08}).


Finally, we make a comment on higher dimensional representations.

\begin{proposition}
In dimensions $d=4$ and higher, every graph can be represented with one convex obstacle.
\end{proposition}

\begin{proof}
Let $G$ be a graph with $n$ vertices.
Consider the moment curve $$\{(t,t^2,t^3,t^4):t \in \mathbf{R}\}.$$ Pick $n$ points $v_i = (t_i, {t_i}^2,{t_i}^3,{t_i}^4)$ on this curve, $i=1,\ldots, n$. The convex hull of these points is a {\em cyclic polytope} $P_n$. The vertex set of $P_n$ is $\{v_1,\ldots, v_n\}$, and any segment connecting a pair of vertices of $P_n$ is an edge of $P_n$ (lying on its boundary). Denote the midpoint of the edge $v_iv_j$ by $v_{ij}$, and let $O$ be the convex hull of the set of all midpoint $v_{ij}$, for which $v_i$ and $v_j$ are not connected by an edge in $G$. Obviously, the points $v_i$ and the obstacle $O$ (or its small perturbation, if we wish to attain general position) show that $G$ admits a representation with a single convex obstacle.
\end{proof}



\section{Open Problems}

The problems we have considered in the last few sections were to ascertain
the obstacle number of graphs when we restrict the kind of obstacles we use, 
namely, general polygons, convex polygons and segments. 
Two other ways to consider the problem would be, firstly, to consider 
restrictions on the placement of the obstacles, and secondly, to consider
restrictions on the kind of graphs we consider. For the first question, 
an interesting problem raised in \cite{AKL09} was to determine graphs which 
require only one obstacle in their outer face.

For the second problem, we realize from Theorem~\ref{ult} that the problem
is more interesting if we consider sparse graphs. 
In \cite{AKL09},
it was shown that outerplanar graphs can be drawn with exactly one obstacle 
in the outer face that was not necessarily convex. Hence, they raised the
question whether outerplanar graphs can be drawn with a finite number of 
convex obstacles. 
 To this, in \cite{fss11} 
it was shown that outerplanar graphs can be drawn with 
only five convex obstacles. Since every tree is an outerplanar graph, this 
also settles the question for trees.
It is an interesting open problem if planar graphs can be drawn 
with a finite number of (convex) obstacles. 

Any graph with $e$ edges can be drawn with $2e$ segment obstacles, by placing
a segment very close to every vertex between any two adjacent edges in the 
drawing. Hence, a sublinear bound on obstacle (or convex obstacle) number 
of planar graphs would also be interesting.

In three dimensions, it is easy to see that every graph can be represented 
with one obstacle. It is interesting, however, to find a bound when we 
restrict ourselves to convex obstacles.

Finally, the upper bound of obstacle numbers is wide open and nothing better 
is known than $2n^2/3$ (this can be achieved since a graph with $e$ edges
needs at most $n(n-1)/2 -e$ obstacles, or $2e$ obstacles from the above
observation). Hence, even a subquadratic bound would be interesting.

%% file: 3_1_combinatorial_games_introduction.tex
\section{Introduction}
A central topic of combinatorial game theory is the study of positional games. The interested reader can find the state of the art methods in Beck's Tic-Tac-Toe book \cite{B}. In general, positional games are played between two players on a {\em board}, the points of which they alternatingly occupy with their marks and whoever first fills a {\em winning set} completely with her/his marks wins the game.
Thus a positional game can be played on any hypergraph, but in this chapter, we only consider 
{\em semi-infinite} games
where all winning sets are finite.
If after countably many steps none of them occupied a winning set, we say that the game ended in a draw.
It is easy to see that we can suppose that the next move of the players depends only on the actual position of the board and is deterministic.\footnote{This is not the case for infinite games and even in semi-infinite games it can happen that the first player can always win the game but there is no $N$ such that the game could be won in $N$ moves. For interesting examples, we refer the reader to the antique papers \cite{ACN, BC, CN}.}
We say that a player has a {\em  winning strategy} if no matter how the other player plays, she/he always wins.
We also say that a player has a {\em drawing strategy} if no matter how the other player plays, she/he can always achieve a draw (or win).
A folklore strategy stealing argument shows that the second player (who puts {\em his} first mark after the first player puts {\em her} first mark, as ladies go first) cannot have a winning strategy, so the best that he can hope for is a draw.
Given any semi-infinite game, either the first player has a winning strategy, or the second player has a drawing strategy. 
We say that the second player can achieve a {\em pairing strategy draw} if there is a matching among the points of the board such that every winning set contains at least one pair. It is easy to see that the second player can now force a draw by putting his mark always on the point which is matched to the point occupied by the first player in the previous step (or anywhere, if the point in unmatched).
Note that in a relaxation of the game for the first player, by allowing her to win if she occupies a winning set (not necessarily first), the pairing strategy still lets the second player to force a draw. 
Such drawing strategies are called {\em strong draws}.
Since in these games only the first player is trying to complete a winning set and the second is only trying to prevent her from doing so, the first player is called {\em Maker}, the second {\em Breaker}, and the game is called a {\em Maker-Breaker} game.

This chapter is about a generalization of the Five-in-a-Row game\footnote{Aka Go-Muku and Am\H oba.} which is the more serious version of the classic Tic-Tac-Toe game. This generalized game is played on the $d$-dimensional integer grid, $\Z^d$, and the winning sets consist of $m$ consecutive gridpoints in $n$ previously given directions. For example, in the Five-in-a-Row game $d=2$, $m=5$ and $n=4$, the winning directions are the vertical, the horizontal and the two diagonals with slope $1$ and $-1$.
Note that we only assume that the greatest common divisor of the coordinates of each direction is $1$, so a direction can be arbitrarily long, e.g.\ $(5,0,24601)$.
The question is, for what values of $m$ can we guarantee that the second player has a drawing strategy?
It was shown by Hales and Jewett \cite{B}, that for the four above given directions of the two dimensional grid and $m=9$ the second player can achieve a pairing strategy draw. In the general version, a somewhat weaker result was shown by Kruczek and Sundberg \cite{KS1}, who showed that the second player has a pairing strategy if $m\ge 3n$ for any $d$. 
They conjectured that there is always a pairing strategy for $m\ge 2n+1$, generalizing the result of Hales and Jewett.\footnote{It is not hard to show that if $m=2n$, then such a strategy might not exist, we show why in Section 3.}

\begin{conjecture}[Kruczek and Sundberg] If $m=2n+1$, then in the Maker-Breaker game played on $\Z^d$, where Maker needs to put at least $m$ of his marks consecutively in one of $n$ given winning directions, Breaker can force a draw using a pairing strategy.
\end{conjecture}

Our main result asymptotically solves their conjecture.

\begin{theorem}\label{one} There is an $m=2n+o(n)$ such that in the Maker-Breaker game played on $\Z^d$, where Maker needs to put at least $m$ of his marks consecutively in one of $n$ given winning directions, Breaker can force a draw using a pairing strategy.
\end{theorem}

In fact we prove the following theorem, which is clearly stronger because of the classical result \cite{H} showing that there is a prime between $n$ and $n+o(n)$.

\begin{theorem}\label{two} If $p=m-1\ge 2n+1$ is a prime, then in the Maker-Breaker game played on $\Z^d$, where Maker needs to put at least $m$ of his marks consecutively in one of $n$ given winning directions, Breaker can force a draw using a pairing strategy.
\end{theorem}

The proof of the theorem is by reduction to a game played on $\Z$ and then using the following recent number theoretic result of Preissmann and Mischler. Later this result was independently rediscovered by Kohen and Sadofschi \cite{KS} and by Karasev and Petrov \cite{KP}, they both gave a short proof using the Combinatorial Nullstellansatz \cite{A}. The latter paper also gives an even shorter topological proof and generalizations.

\begin{lemma}\label{prime}\cite{PM} Given $d_1,\ldots,d_n$ and $p\ge 2n+1$ prime, we can select $2n$ numbers, $x_1,\ldots,x_n,y_1,\ldots,y_n$ all different modulo $p$ such that $x_i+d_i\equiv y_i \mod p$.
\end{lemma}

We prove our theorem in the next section and end the chapter with some additional remarks.

%% file: 3_2_tictactoe.tex
\section{Proof of Theorem \ref{two}}

We consider the winning directions to be the primitive vectors\footnote{A vector $(v_1,\ldots,v_d)\in \Z^d$ is primitive if $gcd(v_1,\ldots,v_d)=1$.} $\vec v_1,...,\vec v_n$. 
Using a standard compactness argument it is enough to show that there is a pairing strategy if the board is $[N]^d$, where $[N]$ stands for $\{1,\ldots,N\}$. For interested readers, the compactness argument is discussed in detail at the end of this section.

First we reduce the problem to one dimension.
Take a  vector $\vec r = (r_1,r_2,...,r_d)$ and transform each grid point $\vec v$ to $\vec v\cdot \vec r$. 
If $\vec r$ is such that $r_j>0$ and $r_{j+1} > N(r_1+\ldots+r_j)$ for all $j$, then this transformation is injective from $[N]^d$ to $\Z$ and each winning direction is transformed to some number, $d_i = |\vec r \cdot \vec v_i|$.\footnote{It is even possible that some of these numbers are zero, we will take care of this later.} 
So we have these $n$ differences, $d_1,\ldots, d_n$, and the problem reduces to avoiding arithmetic progressions of length $m$ with these differences.
From the reduction it follows that if we have a pairing strategy for this game, we also have one for the original.

Let $p$ be a prime such that $2n+1 \le p \le 2n+1+o(n)$. (In \cite{H} it was shown that we can always find such a $p$). 
If we pick a vector $\vec u$ 
uniformly at random from $[p]^d$, 
then for any primitive vector $\vec v$, 
 $\vec u \cdot \vec v$ 
will be divisible by $p$ with probability $1/p$.
Since each winning direction was a primitive vector, 
using the union bound, the probability that at least one of the $\vec u \cdot \vec v_i$ is divisible by $p$ is at most $n/p<1/2$. 
So, there is a $\vec u' = (u_1',u_2',..,u_d') \in [p]^d$ such that none of $\vec u' \cdot \vec v_i$ is divisible by $p$. 
If we now take $\vec r = (r_1,r_2,..,r_d)$ such that $r_j = u_j' + (pN)^{j-1}$, then the dot product with $\vec r$ is injective from $[N]^d$ to $\Z$ and
none of the $d_i = \vec r \cdot \vec v_i$ are divisible by $p$, since $\forall j \ r_j \equiv u_j' \mod p$.  


We now apply Lemma \ref{prime} for $d_1,... , d_n$ to get $2n$ distinct numbers $x_1,x_2,...x_n,y_1,y_2,..,y_n$ such that $0 \le x_i,y_i < p$ and $x_i + d_i \equiv y_i \mod p$. 
Our pairing strategy is, for every $x \equiv x_i \mod p$, $x$ is paired to $x+d_i$ and if $x \equiv y_i \mod p$, then $x$ is paired to $x-d_i$.

To see that this is a good pairing strategy, consider an arithmetic progression $a_1,..., a_m$ of $m=p+1$ numbers with difference, say, $d_i$. Since $p$ and $d_i$ are coprimes, one of the numbers $a_1,..., a_{m-1}$, say $a_j$, must be such that $a_j \equiv x_i \mod p$. Hence $a_j,a_{j+1}$ must be paired in our pairing strategy, showing both cannot be occupied by Maker.\hfill$\Box$\\

For completeness here we sketch how the compactness argument goes. We show that it is sufficient to show that a pairing strategy exists for every finite $[N]^d$ board. For this we use the following lemma.\footnote{We use the version stated in \cite{D}.}

\begin{lemma}\label{kil}\cite{KIL} (K\"onig's Infinity Lemma) Let $V_0,V_1,..$ be an infinite sequence of disjoint non-empty finite sets, and let $G$ be a graph on their union. Assume that every vertex $v$ in a set $V_N$ with $N\ge1$ has a neighbor $f(v)$ in $V_{N-1}$. Then $G$ contains an infinite path, $v_0v_1...$ with $v_N \in V_N$ for all $N$.
\end{lemma}

Given a pairing strategy for $[N_0]^d$, consider a smaller board $[N]^d$ where $N<N_0$. We can think of a pairing strategy as, essentially, a partition of $[N_0]^d$ into pairs and unpaired elements.\footnote{Note that a pairing strategy does not guarantee that every element is paired. It only states that every winning set has a pair. Hence there might be many unpaired elements in a pairing strategy.} We can construct a good pairing strategy for the smaller board by taking the restriction of these set of pairs to $[N]^d$ and leave the elements paired outside $[N]^d$ as unpaired elements. We call this as a restriction of the pairing strategy to the new board. As long as we do not change the length of the winning sets and the prescribed directions, any winning set in the $[N]^d$ board is also a winning set in the $[N_0]^d$ board and hence must have a pair from the restriction. Hence, the Breaker can block all winning pairs and the restriction of the pairing strategy is a valid strategy for Breaker for the smaller board. 

We can now prove the following theorem,
\begin{theorem}\label{compact} Given a fixed set $S,\ |S|=n$, of winning directions,  and positive integer $m$, if Breaker has a pairing strategy for all boards $[N]^d$ 
 and length of winning sets equal to $m$, then Breaker also has a pairing strategy for the $\Z^d$ board. 
\end{theorem}

We will apply K\"onig's Infinity Lemma to prove the theorem. Let $V_N$ be the set of all pairing strategies on the $\{-N,\ldots,N\}^d$ board with winning sets as defined in the theorem. We say a strategy in $V_{N-1}$ and a strategy in $V_N$ have an edge between them if the former is a restriction of the latter. It is easy to see that every vertex in $V_N$ does have an edge to its restriction in $V_{N-1}$. Hence, by the lemma, we must have an infinite path $v_0v_1...$. The union of all these pairing strategies gives a valid pairing strategy for the infinite game.


\section{Possible further improvements and remarks}
As we said before, if $m\le 2n$, then the second player cannot have a pairing strategy draw.
This can be seen as follows. On one hand, in any pairing strategy, from any $m$ consecutive points in a winning direction, there must be at least two points paired to each other in this direction.
On the other hand, there must be a winning direction in which at most $1/n$ of all points are matched to another in this direction.
If we pick a set of size $m-1$ uniformly randomly in this direction, then the expected number of points matched in this direction will be at most $(m-1)/n< 2$.
Thus, there is a set of size $m-1$ that contains only one such point.
Its matching point can now be avoided by extending this set to one way or the other, thereby giving us a winning set with no matched pair.

If $n=1$ or $2$, then a not too deep case analysis shows that the first player has a winning strategy if $m=2n$, even in the {\em strong game}, where the second player also wins if he occupies a winning set.
Moreover, the second player has a pairing strategy for $m=2n+1$ if $n=1$ or $2$, thus, in this case, the conjecture is tight.
However, for higher values, it seems that Breaker can always do better than just playing a pairing strategy, so we should not expect this strategy the best to achieve a draw. Quite tight bounds have been proved for Maker-Breaker games with {\em potential} based arguments, for the latest in generalization of Tic-Tac-Toe games, see \cite{KS2}.
Despite this, from a combinatorial point of view, it still remains an interesting question to determine the best pairing strategy. Unfortunately our proof can only give $2n+2$ (if $2n+1$ is a prime) which is still one bigger than the conjecture.

One could hope that maybe we could achieve a better bound using a stronger result than Lemma \ref{prime} (see for example the conjecture of Roland Bacher in \cite{PM}, whom we would like to thank for directing us to it \cite{Ba}), however, already for $n=3$, our method cannot work. 
Consider the three directions $(1,0),(0,1),(1,1)$. Optimally, we would hope to map them to three numbers, $d_1,d_2,d_3$, all coprime to $6$, such that we can find $x_1,x_2,x_3,y_1,y_2,y_3$ all different modulo $6$ such that $x_i+d_i\equiv y_i \mod 6$. 
But this is impossible since $d_3=d_1+d_2$, so we cannot even fulfill the condition that the differences have to be coprimes to $6$. But even if we forget about that condition, it would still be impossible to find a triple satisfying $d_3=d_1+d_2$. 
If we consider a pairing strategy where the pair of 
any grid point $\vec v$, depends only on $v \cdot r$, then the above argument shows that such a pairing strategy does not exist for the three vectors $(1,0),(0,1),(1,1)$.
However, it is not hard to find a suitable periodic pairing strategy for these three vectors. We would like to end with an equivalent formulation of Conjecture 1.

\begin{conjecture}[Kruczek and Sundberg, reformulated] Suppose we are given $n$ primitive vectors, $\vec v_i$ of $\Z_{2n}^{d}$ for $i\in [n]$. Is it always possible to find a partition of $\Z_{2n}^{d}$ into $\vec x_i^j,\vec y_i^j$ for $i\in [n], j\in [2n]$ such that $\vec x_i^j+\vec v_i=\vec y_i^j$ and $\vec x_i^j-\vec x_i^{j'}$ is not a multiple of $\vec v_i$ for $j\ne j'$?
\end{conjecture}

%% file: 4_appendix.tex
\chapter{Program code}
The following code is in Maple.

\begin{verbatim}
#For accessing log, ceil functions.
with(MTM);

#fmax is a procedure that computes the girth for which a graph on N
#vertices will have the largest supercyle.
#Here, mg denotes the maximum possible girth, max and g will have the
#values of the maximum size of the supercycle and the girth at which 
#it occurs respectively. The procedure returns 2s-2, if this value is
#less than N, we can apply Lemma 2.6 and 2.8 to draw the graphs on N
#vertices.
fmax := proc (N) local g, mg, max, i, exp; 

#Initializations
max := -1;
g := 0;
mg := 2*ceil(evalf(log2((1/3)*N+1))); 

if mg < 3 then RETURN([N, 2*max-2, mg, g]) fi; 

#Main search cycle.
for i from 3 while i <= mg do 
     exp := 2*ceil(evalf(log2((N+1)/i)))+i-1; 
     if max < exp then max := exp; g := i fi
end do; 

RETURN([N, 2*max-2, mg, g]) 
end proc;

seq(fmax(i), i = 6 .. 42, 2);
[6,10,4,3], [8,12,4,4], [10,14,6,5], [12,16,6,6], [14,16,6,6],
[16,16,6,4], [18,16,6,4], [20,18,6,5], [22,20,8,8], [24,20,8,6],
[26,20,8,6], [28,22,8,7], [30,22,8,7], [32,24,8,8], [34,24,8,8], 
[36,24,8,8], [38,24,8,8], [40,24,8,8], [42,24,8,8]
\end{verbatim}